\documentclass[a4paper,oneside]{amsart}

\usepackage{amsmath,amssymb}
\usepackage{tabularx,amscd}
\usepackage{longtable}
\usepackage[arrow,matrix,curve]{xy}
\usepackage{amstext}
\usepackage{amssymb}
\usepackage{amsfonts}
\usepackage{amssymb}
\usepackage{amsmath}
\usepackage{graphicx}
\usepackage{geometry}

\usepackage{tikz}

\newcommand*\mycirc[1]{%
  \begin{tikzpicture}
    \node[draw,circle,inner sep=1pt] {#1};
  \end{tikzpicture}}

\usepackage[latin1]{inputenc}

 \newtheorem{thm}{Theorem}
 \newtheorem{cor}[thm]{Corollary}
 \newtheorem{lem}[thm]{Lemma}
 \newtheorem{prop}[thm]{Proposition}
 \newtheorem{defn}[thm]{Definition}
 \newtheorem{rem}{Remark}
 
 \numberwithin{equation}{section}

\newcommand{\Sc}[2]{\langle #1,#2\rangle}
\newcommand{\Hom}{\mathrm{Hom}}
\newcommand{\End}{\mathrm{End}}
\newcommand{\Ext}{\mathrm{Ext}}
\newcommand{\ses}[3]{0\rightarrow #1\rightarrow #2\rightarrow#3\rightarrow 0}
\newcommand\Rn{\mathbb{R}}

\newcommand\Zn{\mathbb{Z}}
\newcommand\Nn{\mathbb{N}}
\newcommand\Cn{\mathbb{C}}

\begin{document}
\title{Unitarizable representations of quivers}

\author{Thorsten Weist}
\address[Thorsten Weist]{Fachbereich C - Mathematik\\
Bergische Universität Wuppertal\\
D - 42097 Wuppertal, Germany}
\email{weist@math.uni-wuppertal.de}
\author{Kostyantyn Yusenko}
\address[Kostyantyn Yusenko]{Institute of Mathematics  NAS of Ukraine\\
Tereschenkivska str. 3\\
01601 Kiev\\
Ukraine}

\email{kay.math@gmail.com}
\date{\today}
\keywords{Posets \and quivers \and stability \and ADE-classification \and *-algebras \and
rigid tuples}

\begin{abstract}
We investigate representations of $*$-algebras associated
with posets. Unitarizable representations of the corresponding
(bound) quivers (which are polystable representations for some
appropriately chosen slope function) give rise to representations of
these algebras. Considering posets which correspond to unbound quivers this leads to an ADE-classification which describes the unitarization behaviour of their representations. Considering posets which correspond to bound quivers, it is possible to construct unitarizable representations starting with
polystable representations of related unbound quivers which can be glued
together with a suitable direct sum of simple representations.
Finally, we estimate the number of complex parameters parametrizing irreducible unitary non-equivalent representations of the
corresponding algebras.
\end{abstract}

\maketitle
\section*{Introduction}
In the last forty years, finite-dimensional representations of quivers (respectively posets) in the category of vector spaces, also called linear representations in the following, became a huge field of research. It turned out that one can distinguish between (unbound) quivers of representation finite type described and studied by Gabriel in \cite{gab}, tame type described and studied by Nazarova in \cite{Naz73} and the infinite class of quivers of wild type whereas it was proved by Drozd in \cite{dro2} that every representation infinite algebra is either tame or wild. Kleiner and Nazarova  respectively described the posets of representation finite type and representation tame type (see \cite[Theorem 10.1 and Theorem 15.3]{sim}). We refer to Section \ref{reprpos} for precise definitions of (bound) quivers, posets and for the connection between certain representations of (bound) quivers and representations of posets which are all assumed to be finite-dimensional throughout the paper.

During the last years there developed an increasing interest in a full subcategory of the category of linear representations, which is the category of representations of quivers and posets respectively in the category of Hilbert spaces, see for instance \cite{KrugNazRoi,KrugNazRoi2,KrugRoi}. Namely, one can straightforwardly transfer the definitions from the linear to the Hilbert case keeping in mind that the morphisms between two representations should preserve the Hilbert structure, i.e. they are unitary maps. This restriction produces '$*$-wild' problems already in
very simple situations (see \cite{krusa}), i.e. the classification problem contains the classification of two self-adjoint matrices up to unitary isomorphism as a subproblem, see \cite[Chapter 3]{OstrovskyiSamoilenko} for more details. Hence, the idea is to consider representations which
satisfy some additional conditions. That is one of the motivations for studying \textit{orthoscalar representations} of quivers and 
posets, see \cite{KrugNazRoi,KrugRoi} and Section \ref{staralg} for a precise definition. It makes the subject even more interesting that there are lots of topics which are closely related to such representations. For instance such representations are connected to Hermann Weyl's problem of describing the spectra of the sum of two Hermitian matrices and its generalizations, see \cite{Fulton} and Section \ref{staralg} for a precise statement of the problem.

To study orthoscalar representations it is convenient to use the language of $*$-algebras and their $*$-representations (see Section \ref{staralg} for the definition of $*$-algebras). With a poset $\mathcal{N}$ consisting of $n$ elements and a fixed tuple $\chi=(\chi_0;\chi_1,\ldots,\chi_n) \in \mathbb R_+^{n+1}$, called \textit{weight} in what follows, we associate the following $*$-algebra over the complex numbers 
$$
    \mathcal A_{\mathcal N,\chi}=\mathbb C\left\langle p_1,\ldots,p_n\ \Bigg |
    \begin{array}{l}
           p_i=p_i^*=p_i^2 \\
           p_jp_i=p_ip_j=p_i,\quad i\prec j \\
           \chi_1p_1+\ldots+\chi_n p_n=\chi_0 e
    \end{array} \right \rangle,
$$
where $e$ denotes the identity element. Two problems naturally arise: to find those weights $\chi$ for which the algebra $\mathcal A_{\mathcal N,\chi}$ has at least one non-zero representation; for
appropriate $\chi$, to describe all irreducible $*$-representations up to
unitary equivalence.  The second problem could turn out to be
''hopeless", i.e. the isomorphism classes of $*$-representations can
depend on arbitrarily many continuous parameters, or the algebra can
be even $*$-wild.
 
In this paper we only consider finite-dimensional representations of $\mathcal A_{\mathcal N,\chi}$.  Whereas every representation of $\mathcal A_{\mathcal N,\chi}$ canonically defines a representation of $\mathcal N$ in the category of vector spaces, it is not clear which kind of linear representations of $\mathcal N$ allow a choice of a Hermitian metric such that they give rise to $*$-representations of $\mathcal A_{\mathcal N,\chi}$ (those objects that possess such a choice are called \textit{unitarizable} following \cite{KrugRoi}). One of the essential parts of this paper is devoted to this problem.
Namely, we say that a representation $(V_0;(V_i)_{i\in\mathcal{N}})$ of $\mathcal{N}$ can be \textit{unitarized} with the weight $\chi$ if it possesses a choice of a Hermitian structure in $V_0$ in such a way that for the corresponding
projections $P_i:V_0\rightarrow V_i$ the following equality holds
$$
    \chi_1P_1+\ldots+\chi_n P_n=\chi_0 \mathbb I.
$$

We use the fact that unitarizable representations of posets can be identified with polystable representations of quivers, i.e. representations which can be decomposed into stable ones of the same slope, going back to King's work \cite{kin}, see Section \ref{rep} for more details. This approach turns out to be very useful because lots of statements which are known for stable quiver representations can be used and easily transferred to representations of posets.\\After recalling the notion of general representations of quivers in Section \ref{unitgen} as introduced by Schofield, see \cite{sch}, we state our first result, Theorem \ref{thm1}, saying that each general Schurian representation of a poset corresponding to an unbound quiver can be unitarized with a certain weight and hence gives rise to a $*$-representation of $\mathcal A_{\mathcal N,\chi}$. Recall that Schurian representations are representations whose endomorphism ring consists only of elements of the base field. Moreover, we specify those weights $\chi$ in terms of the dimension vector of the corresponding representation. This result is later used to show that the corresponding Hermitian operators are rigid in
the sense of N.Katz in \cite{Katz96}, see Section 3.2 for a precise definition and Theorem \ref{rigidprimitive} for the statement.\\\\
In Section \ref{boundquivers} we consider quivers (which are also related to posets) which are bound by some ideal as recalled in Section \ref{reprpos}. Actually, this covers plenty of interesting cases of quivers which were not considered in \cite{KrugNazRoi,KrugRoi}. On the one hand Section \ref{boundquivers} gives an explicit construction of stable representations of bound quivers including the description of the linear form defining the stability condition. On the other hand, since the constructed representations are unitarizable representations of the corresponding posets it is again possible to understand them as representations in the category of Hilbert spaces.  More detailed, starting with a polystable representation of a related unbound quiver we glue its direct summands together with a direct sum of simple representations, which corresponds to the vertex where the relations determining the ideal start, in such a way that the resulting representations are stable and satisfy the relations corresponding to the ideal. This result can also be seen as a starting point of a classification of unitarizable representations of non-primitive posets, see Remark \ref{bem} for more details. Recall that a primitive poset is the cardinal sum of linearly ordered sets, see also Section \ref{reprpos} for more details.\\\\
We start Section \ref{unitnon} by illustrating this method on examples and finish it stating the next result which classifies posets corresponding to unbound quivers due to their unitarization behaviour. To do so we make also use of results of \cite{hp,sch}. More detailed we state that if the poset is of representation finite type, then every indecomposable representation is unitarizable, see \cite{GrushevoyYusenko}, if it is primitive and of representation tame type then every Schurian representations is unitarizable, which essentially follows from \cite{hp}. What we show is that in the representation wild case, in addition to the unitarizable representations described in Theorem \ref{thm1}, there are also
non-unitarizable Schurian representations which depend on an arbitrary number of complex parameters. It is remarkable that the classification mentioned in the beginning of the introduction also appears in this situation.\\ 
Let us remark that it is still an open question whether or not each Schurian representation
of a non-primitive tame poset is unitarizable. Nevertheless, the
article gives an idea how to handle these cases. Thereby, the main idea is to extend some non-primitive subposet and its representations appropriately using the methods of Section \ref{boundquivers}.\\\\
The main result of the last section concerns the number of parameters that
parametrize non-equivalent irreducible $*$-representations of
$\mathcal A_{\mathcal N,\chi}$. If $\mathcal N$ is primitive and of
representation finite type, then it is known that one discrete parameter
parametrizes all irreducible nonequivalent representations (see \cite{KrugPopSam}) and that if $\mathcal
N$ is primitive and of representation tame type, it is at most one continuous
parameter depending on the weight $\chi$ (see \cite{albostsam} and references therein). We show that if
$\mathcal N$ is of representation wild type (primitive or non-primitive), then there exists
a weight $\chi_{\mathcal N}$ such that there are families of
non-equivalent irreducible $*$-representations of $\mathcal
A_{\mathcal N,\chi_{\mathcal N}}$ which depend on arbitrary many continuous
parameters; conjecturally such algebras are of $*$-wild representation type.

\section*{Acknowledgements}
The first author likes to thank Klaus Bongartz and Markus Reineke
for very helpful discussions and comments. The second author thanks
the Hausdorff Institute (Bonn, Germany), Chalmers University of
Technology (Gothenburg, Sweden) for hospitality and also Vasyl
Ostrovskiy, Yurii Samoilenko and  Ludmila Turowska for helpful
remarks. The second author was partially supported by the Swedish
Institute and DFG grant SCHM1009/4-1. This paper was initiated
during a visit of the second author to the University of Wuppertal
(Germany), the hospitality of the Algebra and Number Theory Group is
gratefully acknowledged.

The article will appear in Algebras and Representation Theory. The final publication is available at springerlink.com, DOI: 10.1007/s10468-012-9360-4.

\section{Preliminaries} \label{Section1}

\subsection{Representations of quivers and posets}\label{reprpos}
Let $Q$ be a finite quiver which is given by a set of vertices $Q_0$ and a set of arrows $Q_1$ denoted by $\rho:q\rightarrow q'$ for $q,q'\in Q_0$. The vertex $q$ is called {\it tail}, and the vertex $q'$ is called {\it head} of the arrow $\rho$. A vertex $q\in Q_0$ is called {\it sink} if there does not exist an arrow $\rho:q\rightarrow q'\in Q_1$. A vertex $q\in Q_0$ is called {\it source} if there does not exist an
arrow $\rho:q'\rightarrow q\in Q_1$.\\
For a
vertex $q\in Q_0$ let
\[N_q=\{q'\in Q_0\mid \exists \rho:q\rightarrow q'\vee \exists \rho:q'\rightarrow q\}\]
be the set of its neighbours.\\
In the following we only consider finite quivers without oriented cycles.
Define the abelian group
\[\mathbb{Z}Q_0=\bigoplus_{q\in Q_0}\mathbb{Z}q\] and the monoid of dimension vectors $\mathbb{N}Q_0$.\\
Let $k$ be an algebraically closed field. A finite-dimensional $k$-representation of $Q$ is given by a tuple
\[X=((X_q)_{q\in Q_0},(X_{\rho})_{\rho \in Q_1}:X_q\rightarrow X_{q'})\]
of finite-dimensional $k$-vector spaces and $k$-linear maps between
them. A morphism of representations $f:X\rightarrow Y$ is a tuple $f=(f_q:X_q\rightarrow Y_q)_{q\in Q_0}$ of $k$-linear maps such that $Y_{\rho}f_q=f_{q'}X_{\rho}$ for all $\rho:q\rightarrow q'$. We denote by $\mathrm{Rep}(Q)$ the abelian category of finite-dimensional representations of $Q$. 

We call a representation $X$ \textit{Schurian} if $\End(X):=\Hom(X,X)=k$. We say that $X$ is \emph{strict} if all maps $X_{\rho}$ are
injective. The dimension vector $\underline{\dim}X\in\mathbb{N}Q_0$
of $X$ is defined by
\[\underline{\dim}X=\sum_{q\in Q_0}\dim X_q\cdot q.\]
Let $\alpha \in\mathbb{N}Q_0$ be a dimension vector. The variety
$R_\alpha(Q)$ of $k$-representations of $Q$ with dimension vector
$\alpha$ is defined as the affine $k$-space
\[R_\alpha(Q)=\bigoplus_{\rho:q\rightarrow q'} \mathrm{Hom}( k^{\alpha_q},k^{\alpha_{q'}}).\]
The algebraic group $G_\alpha=\prod_{q\in Q_0} \textrm{Gl}_{\alpha_q}(k)$
acts on $R_\alpha(Q)$ via simultaneous base change, i.e.
\[(g_q)_{q\in Q_0}\ast
(X_{\rho})_{\rho\in Q_1}=(g_{q'}X_{\rho}g_{q}^{-1})_{\rho:q\rightarrow
q'}.\] The orbits are in bijection with the isomorphism classes of
$k$-representations of $Q$ with dimension vector $\alpha$.

Let $kQ$ be the path algebra of $Q$ and let $RQ$ be the arrow ideal, see \cite[Chapter II.1]{ass} for a definition. A relation in $Q$ is a $k$-linear combination of paths of length at
least two which have the same head and tail. For a set of relations
$(r_j)_{j\in J}$ we can consider the admissible ideal $I$ generated
by these relations, where admissible means that we have $RQ^m\subseteq
I\subseteq RQ^2$ for some $m\geq 2$. Now a representation $X$ of $Q$
is bound by $I$, and thus a representation of the bound quiver
$(Q,I)$, if $X_{r_j}=0$ for all $j\in J$. For every dimension vector
this defines a closed subvariety of $R_\alpha(Q)$ denoted by
$R_\alpha(Q,I)$. If $R$ is a minimal set of relations generating
$I$, by $r(q,q',I)$ we denote the number of relations with starting
vertex $q$ and terminating vertex $q'$. Following \cite{bon}, for the
dimension of $R_\alpha(Q,I)$ we get
\[\dim R_\alpha(Q,I)\geq\dim R_\alpha(Q)-\sum_{(q,q')\in (Q_0)^2}r(q,q',I)\alpha_q \alpha_{q'}.\]

Let $C_{(Q,I)}$ be the Cartan matrix of $(Q,I)$, i.e. $c_{q',q}=\dim
e_q(kQ/I)e_{q'}$ where $e_q$ denotes the primitive idempotent (resp.
the trivial path) corresponding to the vertex $q$. On $\Zn Q_0$ a
non-symmetric bilinear form, the Euler characteristic, is defined by
\[\Sc{\alpha}{\beta}:=\alpha^T(C_{(Q,I)}^{-1})^T\beta.\]
Then for two representation $X$ and $Y$ we have
\[\Sc{X}{Y}:=\Sc{\underline{\dim} X}{\underline{\dim} Y}=\sum_{i=0}^{\infty}(-1)^i\dim\Ext^i(X,Y),\]
see for instance \cite[Proposition 3.13]{ass}.
Moreover, if $Q$ is unbound, for two representations $X$, $Y$ of $Q$ with
$\underline{\dim} X=\alpha$ and $\underline{\dim} Y=\beta$ we have
\[\Sc{X}{Y}=\dim\Hom(X,Y)-\dim\Ext(X,Y)=\sum_{q\in Q_0}\alpha_q \beta_q-\sum_{\rho:q\rightarrow q'\in Q_1}\alpha_q\beta_{q'}\]
and $\Ext^i(X,Y)=0$ for $i\geq 2$, see \cite[Section 2]{rin}.\\ As usual we call a dimension vector $\alpha\in\Nn Q_0$ a root of the quiver $Q$ if there exists an indecomposable representation of $Q$ with $\underline{\dim}X=\alpha$. A root is called real if $\Sc{\alpha}{\alpha}=1$ and imaginary otherwise. \\

Let $X$ and $Y$ be two representations of a quiver $Q$. Then we can
consider the linear map
\[\gamma_{X,Y}:\bigoplus_{q\in Q_0}\Hom(X_q,Y_q)\rightarrow\bigoplus_{\rho:q\rightarrow q'\in Q_1}\Hom(X_q,Y_{q'})\]
with $\gamma_{X,Y}((f_q)_{q\in Q_0})=(Y_{\rho}f_q-f_{q'}X_{\rho})_{\rho:q\rightarrow q'\in Q_1}$.\\
We have $\ker(\gamma_{X,Y})=\Hom(X,Y)$ and
$\mathrm{coker}(\gamma_{X,Y})=\Ext(X,Y)$, see \cite[Section 2]{rin}. The first
statement is obvious. The second one follows because every exact
sequence $E(f)\in\Ext(X,Y)$ is defined by a morphism
$f\in\bigoplus_{\rho:q\rightarrow q'\in Q_1}\Hom(X_q,Y_{q'})$ in the
following way
\[\ses{Y}{((Y_q\oplus X_q)_{q\in Q_0},(\begin{pmatrix}Y_{\rho}&f_{\rho}\\0&X_{\rho}\end{pmatrix})_{\rho\in Q_1})}{X}\]
with the canonical inclusion on the left hand side and the canonical projection on the right hand side. Now it is straightforward to check that two sequences $E(f)$ and $E(g)$ are equivalent if and only if $f-g\in\mathrm{Im}(\gamma_{X,Y})$.\\
As far as bound quivers are concerned, we just have to consider those
exact sequences such that the middle term also satisfies the
relations, thus we have $\Ext_{(Q,I)}(X,Y)\subseteq\Ext_{Q}(X,Y)$.\\\\

Let $\mathcal N$ be a finite poset where $\prec$ denotes the
partial order in $\mathcal N$. 
A finite-dimensional representation of $\mathcal{N}$ is given by a
collection of finite-dimensional $k$-vector spaces
$$V=(V_0;(V_q)_{q\in\mathcal{N}})$$
such that $V_q\subseteq V_0$ for all
$q\in\mathcal{N}$ and $V_q\subseteq V_{q'}$ if $q\prec q'$. A morphism
between two representations $(V_0;(V_q)_{q\in\mathcal{N}})$ and
$(W_0;(W_q)_{q\in\mathcal{N}})$ is given by a $k$-linear map
$g:V_0\rightarrow W_0$ such that $g(V_q)\subseteq W_q$ for all
$q\in\mathcal{N}$. By $\mathrm{Rep}(\mathcal{N})$ we denote the additive category of representations of a poset $\mathcal{N}$.

Denote by $\mathcal N^0$ the extension of $\mathcal N$ by a unique maximal element $q_0$. With $\mathcal N$ we associate the Hasse quiver of $\mathcal N^0$ which will be denoted by $Q(\mathcal N)$, i.e. we orient all edges of the Hasse diagram of the poset $\mathcal{N}$ to the unique maximal element. The other way around let $Q$ be a connected quiver without oriented cycles and multiple arrows.
Moreover, we assume that all arrows are oriented to one vertex $q_0$ which
is called the root. To $Q$ we can naturally associate the poset $\mathcal{N}(Q)=Q_0\backslash \{q_0\}$ such that $q\prec q'$ if and only if there exists a path from $q$ to $q'$.
A poset is said to be \emph{primitive} if it is the disjoint (cardinal) sum of
linearly ordered sets $L_i$ of order $n_i$. In this case we
denote the poset and corresponding quiver by $(n_1,\ldots,n_s)$. Since the quiver $Q(\mathcal N)$ corresponding to a primitive poset $\mathcal N$ has a star-shaped form, it is called \textit{star-shaped}.

We recall the classification of posets by their representation type (see for example \cite[Chapter 10 and Chapter 15]{sim} for precise definitions of the terms finite, tame and wild representation type).

\begin{thm} (Kleiner \cite[Theorem 10.1]{sim} and Nazarova \cite[Theorem 15.3]{sim}) \label{reptype}
\begin{enumerate}
\item
A poset $\mathcal N$ is of representation finite type if and only if the quiver $Q(\mathcal N)$ does not contain any of the following critical quivers 
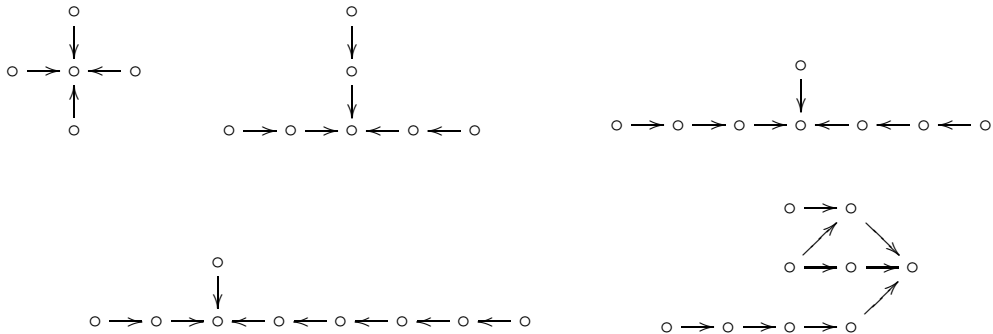
\begin{center}
$\xymatrix @-1pc { &{\circ}\ar@{->}[d]&\\
            {\circ} \ar@{->}[r]&{\circ}\ar@{<-}[d]\ar@{<-}[r]&{\circ}\\
            &{\circ}&}$
            \qquad      
$\xymatrix @-1pc { &&{\circ}\ar@{->}[d]&&\\
            &&{\circ}\ar@{->}[d]&&\\
            {\circ}\ar@{->}[r]&{\circ}\ar@{->}[r]&{\circ}\ar@{<-}[r]&{\circ}\ar@{<-}[r]&{\circ}&}$
\qquad
$\xymatrix @-1pc {\\&&&{\circ}\ar@{->}[d]&&\\
            {\circ}\ar@{->}[r]&{\circ}\ar@{->}[r]&{\circ}\ar@{->}[r]&{\circ}\ar@{<-}[r]&{\circ}\ar@{<-}[r]&{\circ}\ar@{<-}[r]&{\circ}}$
\bigskip 

$\xymatrix @-1pc {\\&&{\circ}\ar@{->}[d]&&\\
            {\circ}\ar@{->}[r]&{\circ}\ar@{->}[r]&{\circ}\ar@{<-}[r]&{\circ}\ar@{<-}[r]&{\circ}\ar@{<-}[r]&{\circ}\ar@{<-}[r]&{\circ}\ar@{<-}[r]&{\circ}}$
\qquad
$\xymatrix @-1pc
            {&&&{\circ}\ar@{->}[r]&{\circ}\ar@{->}[dr]\\
             &&&{\circ}\ar@{->}[r]\ar@{->}[ur]&{\circ}\ar@{->}[r]&{\circ}\\
             &{\circ}\ar@{->}[r]&{\circ}\ar@{->}[r]&{\circ}\ar@{->}[r]&{\circ}\ar@{->}[ur]}$
\end{center}
 as a proper subquiver.

\item A poset $\mathcal N$ is of representation tame type if and only if the quiver $Q(\mathcal N)$ does not contain any of the following critical quivers 

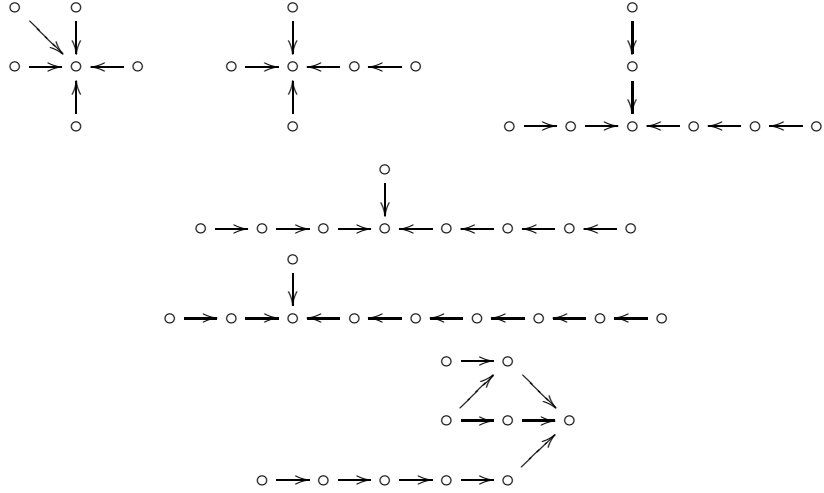
\begin{center}
$\xymatrix @-1pc { {\circ} \ar@{->}[dr]&{\circ}\ar@{->}[d]&\\
            {\circ} \ar@{->}[r]&{\circ}\ar@{<-}[d]\ar@{<-}[r]&{\circ}\\
            &{\circ}&}$
            \qquad
$\xymatrix @-1pc { &{\circ}\ar@{->}[d]&\\
            {\circ} \ar@{->}[r]&{\circ}\ar@{<-}[d]\ar@{<-}[r]&{\circ}\ar@{<-}[r]&{\circ}\\
            &{\circ}&}$
            \qquad
$\xymatrix @-1pc { &&{\circ}\ar@{->}[d]&&\\
            &&{\circ}\ar@{->}[d]&&\\
            {\circ}\ar@{->}[r]&{\circ}\ar@{->}[r]&{\circ}\ar@{<-}[r]&{\circ}\ar@{<-}[r]&{\circ}\ar@{<-}[r]&{\circ}}$

\medskip

$\xymatrix @-1pc {&&&{\circ}\ar@{->}[d]&&\\
            {\circ}\ar@{->}[r]&{\circ}\ar@{->}[r]&{\circ}\ar@{->}[r]&{\circ}\ar@{<-}[r]&{\circ}\ar@{<-}[r]&{\circ}\ar@{<-}[r]&{\circ}\ar@{<-}[r]&{\circ}}$
\qquad
$\xymatrix @-1pc {&&{\circ}\ar@{->}[d]&&\\
            {\circ}\ar@{->}[r]&{\circ}\ar@{->}[r]&{\circ}\ar@{<-}[r]&{\circ}\ar@{<-}[r]&{\circ}\ar@{<-}[r]&{\circ}\ar@{<-}[r]&{\circ}\ar@{<-}[r]&{\circ}\ar@{<-}[r]&{\circ}}$

\medskip
$\xymatrix @-1pc
            {&&&{\circ}\ar@{->}[r]&{\circ}\ar@{->}[dr]\\
             &&&{\circ}\ar@{->}[r]\ar@{->}[ur]&{\circ}\ar@{->}[r]&{\circ}\\
             {\circ}\ar@{->}[r]&{\circ}\ar@{->}[r]&{\circ}\ar@{->}[r]&{\circ}\ar@{->}[r]&{\circ}\ar@{->}[ur]}$

\end{center}
as a proper subquiver.

\end{enumerate}
\end{thm} 

In the following, the two non-primitive posets in the previous theorem will be denoted by $(N,4)$ and $(N,5)$ respectively. 

We briefly recall the relation between representations of posets and representations of bound quivers. Everything presented here is well-known, see for instance \cite{dro} for a more general setup.\\
Let $Q(\mathcal{N})$ be the quiver induced by a poset $\mathcal{N}$. Let $\alpha \in\Nn Q(\mathcal{N})_0$ be a dimension vector.
By $S_\alpha(Q(\mathcal{N}))\subset R_\alpha(Q(\mathcal{N}))$ we denote the (possibly empty) open subvariety of
strict representations. For every (non-oriented) cycle
$\rho_1\ldots\rho_n\tau_m^{-1}\ldots\tau_1^{-1}$ with
$\rho_i,\,\tau_j\in Q(\mathcal{N})_1$ and $\rho_i\neq\tau_j$ we define a relation
\[\rho_1\ldots\rho_n-\tau_1\ldots\tau_m.\]
Let $I$ be the ideal generated by all such relations.

Let $V=(V_0;(V_q)_{q \in \mathcal N})$ be a representation of $\mathcal N$ with dimension vector
$\alpha$. This defines a representation $F(V)\in S_\alpha(Q(\mathcal{N}),I)$
satisfying the stated relations. Indeed, every inclusion $V_q\subset
V_{q'}$ defines an injective map $F(V)_{\rho_{q,q'}}:V_q\rightarrow
V_{q'}$. For two
arbitrary representations $V$ and $W$ a morphism $g:V\rightarrow W$,
defines a morphism $F(g):F(V)\rightarrow F(W)$ where
$F(g)_q:=g|_{V_q}:F(V)_q\rightarrow F(W)_q$.

The other way around let $X\in S_\alpha(Q(\mathcal{N}),I)$. This gives rise to a
representation $G(X)$ of $\mathcal{N}$ by defining
$G(X)_q=X_{\rho^q_n}\circ\ldots\circ X_{\rho_1^q}(X_q)$ for some
path $\rho_1^q\ldots\rho_n^q$ from $q$ to $q_0$.
This definition is independent of the chosen path. Moreover, every morphism $f=(f_q)_{q\in Q_0}:X\rightarrow Y$ defines a morphism $G(f)$ which is induced by $f_{q_0}:X_0\rightarrow Y_0$.\\
Thus we get an equivalence between the categories of strict
representations of $Q(\mathcal{N})$ bound by $I$ and representations of
$\mathcal{N}$. This equivalence also preserves dimension vectors.

If the global dimension of $kQ(\mathcal{N})/I$ is at most two, see for instance \cite[Chapter A.4]{ass} for a definition, for two
representations $X$ and $Y$ with $\underline{\dim} X=\alpha$ and
$\underline{\dim} Y=\beta$ we get
\begin{eqnarray*}\Sc{X}{Y}&=&\dim\Hom(X,Y)-\dim\Ext^1(X,Y)+\dim\Ext^2(X,Y)\\&=&\sum_{q\in Q(\mathcal{N})_0}\alpha_q\beta_q-\sum_{\rho:q\rightarrow q'\in Q(\mathcal{N})_1}\alpha_q\beta_{q'}+\sum_{(q,q')\in (Q(\mathcal{N})_0)^2}r(q,q',I)\alpha_q\beta_{q'},\end{eqnarray*}
see \cite{bon}. This defines a quadratic form $q_{Q(\mathcal{N})}(\alpha):=\Sc{\alpha}{\alpha}$, often called Tits form (in some cases it coincides with the Drozd form for posets as introduced in \cite{dro}, see also \cite{sim2} for the connection between different quadratic forms associated with posets).\\
In order to shorten notation and because we are also only interested in bound representation of $Q(\mathcal{N})$ if it has unoriented cycles, we denote by $Q(\mathcal{N})$ the quiver $Q(\mathcal{N})$ bound by $I$ as constructed above. In particular, if $\mathcal{N}$ contains pairwise disjoint elements $q_1,q_2,q_3$ such that $q_1\prec q_2$ and $q_1\prec q_3$, the quiver $Q(\mathcal{N})$ is bound, otherwise it is unbound.
\subsection{Orthoscalar representations of posets and quivers. $*$-Algebras associated to posets and graphs.}\label{staralg}

Let $k=\Cn$. Fix a poset $\mathcal{N}$ and the corresponding quiver $Q(\mathcal{N})$. Let $\mathcal{H}$ be the category of (finite-dimensional) Hilbert spaces. Consider the subcategory $\mathrm{Rep}(Q(\mathcal{N}),\mathcal{H})$ of $\mathrm{Rep}(Q(\mathcal{N}))$ consisting of representations $X$ such that $X_q$ are Hilbert spaces for all $q\in Q_0$. Denoting by $X_{\rho}^*$ the adjoint linear map of $X_{\rho}$ for a morphism $f:X\rightarrow Y$ we additionally require that $f_qY^{*}_{\rho}=X_{\rho}^*f_{q'}$ for all $\rho:q\rightarrow q'$. It is straightforward to check (see for example \cite[Chapter 1]{OstrovskyiSamoilenko} for a similar statement) that two representations $X$ and $Y$ are isomorphic in $\mathrm{Rep}(Q(\mathcal{N}),\mathcal{H})$, i.e. there exists an invertible morphism $f:X\rightarrow Y$, if and only if they are unitary isomorphic, i.e. $f_{\rho}$ is a unitary linear map for every $\rho\in Q(\mathcal{N})_1$. Now we may easily transfer this definition to $\mathrm{Rep}(\mathcal{N})$ and form the category $\mathrm{Rep}(\mathcal{N},\mathcal{H})$. 
Even in simple cases the description of indecomposable objects in the category 
$\mathrm{Rep}(\mathcal{N},\mathcal{H})$ is a very hard, so-called $*$-wild, problem \cite{krusa}. Thus it is natural to consider subcategories of $\mathrm{Rep}(\mathcal{N},\mathcal{H})$. 

We say that an object $V=(V_0,(V_q)_{q\in\mathcal{N}})$ of $\mathrm{Rep}(\mathcal{N},\mathcal{H})$ is \textit{orthoscalar} if there exists a \textit{weight} $\chi=(\chi_0,(\chi_q)_{q\in \mathcal N})\in\Rn^{|\mathcal{N}|+1}$ such that 
\begin{equation}
\sum_{q \in \mathcal N}\chi_qP_q=\chi_0P_0 \label{orthorelation}
\end{equation}
where the $P_i$ is the orthogonal projection onto the subspace $V_i$. Denote this category, which is a full subcategory of $\mathrm{Rep}(\mathcal{N},\mathcal{H})$, by $\mathrm{Rep}(\mathcal{N},\mathcal{H})_{\mathrm{os}}$. We should mention that the term \textit{locally-scalar} instead of orthoscalar also appears in the literature (see \cite{KrugRoi}). Note that if $\alpha\in\Nn^{|\mathcal{N}|+1}$ is the dimension vector of $V$ by taking the trace on both sides of \eqref{orthorelation} we obtain
\[\sum_{q\in\mathcal{N}}\chi_q\alpha_q=\chi_0\alpha_0.\]

Every object of $\mathrm{Rep}(\mathcal{N},\mathcal{H})_{\mathrm{os}}$ can be identified with an object of $\mathrm{Rep}(\mathcal{N})$ applying the forgetful functor. Thus it is natural to ask for which kind of representations $V=(V_0,(V_q)_{q\in\mathcal{N}})$ we can choose a hermitian form such that there exists a weight $\chi=(\chi_0,(\chi_q)_{q\in\mathcal{N}}) \in\Rn^{|\mathcal{N}|+1}$ such that the corresponding orthoprojections $P_q$ onto subspaces $V_q$ satisty \eqref{orthorelation}.
In this case, following \cite[Section 4]{KrugRoi} we say that $V$ is \textit{unitarizable} with (or can be \textit{unitarized} with) the weight $\chi$.

Recall that a \textit{$*$-algebra} is an algebra $\mathcal A$ over $\Cn$ together with an anti-automorphism $~^*:\mathcal A\rightarrow \mathcal A$, i.e. $~^{*}$ is an algebra automorphism such that $(ab)^*=b^{*}a^{*}$ and $(a^*)^*=a$ for all $a,b\in \mathcal A$.

Assume that $\mathcal N$ consists of $n$ points. For a given weight $\chi=(\chi_0;\chi_1,\ldots,\chi_n) \in \mathbb
R_+^{n+1}$ consider the $*$-algebra defined by 
$$
    \mathcal A_{\mathcal N,\chi}=\mathbb C\left\langle p_1,\ldots,p_n\quad \Bigg |
    \begin{array}{l}
           p_i=p_i^*=p_i^2 \\
           \chi_1p_1+\ldots+\chi_n p_n=\chi_0 e \\
           p_jp_i=p_ip_j=p_i,\quad i\prec j
    \end{array} \right \rangle
$$
where $e$ denotes the identity element of $\mathcal A_{\mathcal N,\chi}$.

The objects of $\mathrm{Rep}(\mathcal{N},\mathcal{H})_{\mathrm{os}}$ correspond to finite-dimensional $*$-representations of $\mathcal A_{\mathcal N,\chi}$ and results about the structure of these objects can be formulated in terms of representations of $\mathcal A_{\mathcal N,\chi}$. Let us describe how these algebras are related to $*$-algebras associated with star-shaped graphs (considered for example in \cite{albostsam,KrugPopSam}) in the case when the poset is primitive.

Let $\Gamma=(\Gamma_0,\Gamma_1)$ be a connected graph with vertices $\Gamma_0$ and edges $\Gamma_1$. 
For a given poset $\mathcal N$ by $\Gamma(\mathcal N)$ we denote the underlying graph of the quiver $Q(\mathcal N)$. 
We call $\Gamma$ \textit{star-shaped} if it is the underlying graph of a star-shaped quiver.
Clearly $\Gamma(\mathcal N)$ is star-shaped if and ony if $\mathcal N$ is primitive.  

Assuming that the graph $\Gamma$ is of the type $(m_1,\ldots,m_n)$ we identify the set of vertices $\Gamma_0$ with $(g_0;g_i^{(j)})$, where $g_0$ is the root
vertex and $g_{i_1}^{(j)}$ and $g_{i_2}^{(j)}$ lie on the same
branch of $\Gamma$. Fixing some vector $\omega=(\omega_0;\omega_{i}^{(j)}) \in \mathbb
R^{|\Gamma_0|}_+$ with
$\omega_{i_1}^{(j)}>\omega_{i_2}^{(j)}$ if $i_1>i_2$ (following \cite{albostsam,KrugPopSam} a vector with such properties is called \textit{character}), we consider the $*$-algebra
$$
    \mathcal B_{\Gamma,\omega}=\mathbb C\left\langle a_1,\ldots,a_n\quad \Bigg |
    \begin{array}{l}
           a_i=a_i^* \\
          (a_i-\omega_1^{(i)})\ldots (a_i-\omega_{m_i}^{(i)})=0 \\
          a_1+\ldots+a_n=\omega_0 e
    \end{array} \right \rangle.
$$
Any $*$-representation of $\mathcal B_{\Gamma,\omega}$ in some
Hilbert space is given by an $n$-tuple of Hermitian operators with
spectra $\sigma(A_i)\in
\{\omega_1^{(i)}<\ldots<\omega_{m_i}^{(i)}\}$ such that
\begin{equation*} 
    A_1+\ldots+A_n=\omega_0 \mathbb I.
\end{equation*}
Recall that the last equation is connected with generalizations of Hermann Weyl's problem: can one describe the eigenvalues of the sum of two Hermitian $n\times n$-matrices in terms of the eigenvalues of the two single matrices, see also \cite{Fulton} for the description of the classical problem of Hermann Weyl and generalizations. Note that Klyachko, see \cite{kly}, solved a more general version of this problem. The interested reader should also consult \cite{albostsam,KrugPopSam} and references therein for generalizations. Fixing a
finite-dimensional representation of $\mathcal B_{\Gamma,\omega}$ in
some Hilbert space $H$, for each operator $A_i$ we can consider its
spectral decomposition
$$
    A_i=\omega_1^{(i)} \widetilde P_1^{(i)}+\ldots+\omega_{m_i}^{(i)}
    \widetilde P_{m_i}^{(i)}.
$$
If the poset $\mathcal N$ is primitive of type $(m_1,\ldots,m_n)$, then each $*$-representation
of $\mathcal A_{\mathcal N,\chi}$ generates a $*$-representation of
$\mathcal B_{\Gamma(\mathcal N),\omega}$ for some character $\omega$ which can be written in terms of the weight $\chi$. More precisely, let $(P_i^{(j)})$ be a
$*$-representation of $\mathcal A_{\mathcal N,\chi}$, which means
that
$P_{i_1}^{(j)}P_{i_2}^{(j)}=P_{i_2}^{(j)}P_{i_1}^{(j)}=P_{i_1}^{(j)}$
if $i_1<i_2$ and
$$
    \chi_1^{(1)} P_1^{(1)}+\ldots+\chi_{m_1}^{(1)}
    P_{m_1}^{(1)}+\ldots+\chi_1^{(n)} P_1^{(n)}+\ldots+\chi_{m_n}^{(n)}
   P_{m_n}^{(n)}=\chi_0 \mathbb I.
$$
Letting $\widetilde P_{1}^{(j)}=P_{1}^{(j)}$,
$\widetilde P_{i}^{(j)}=P_{i}^{(j)}-P_{i-1}^{(j)}$ and taking the weight
$\omega_{m_j}^{(j)}=\chi_{m_j}^{(j)}$,
$\omega_{i}^{(j)}=\chi_{i}^{(j)}+\omega_{i+1}^{(j)}$,
$\omega_0=\chi_0$, we get a representation of $\mathcal
B_{\Gamma(\mathcal N),\omega}$. Note that one can prove that $\mathcal
A_{\mathcal N,\chi}$ and $\mathcal B_{\Gamma(\mathcal N),\omega}$ are isomorphic
using the same transformation between the projections.

\subsection{Stable representations and unitarizable representations of quivers}\label{rep}

In order to study representations of the algebra $\mathcal A_{\mathcal N,\chi}$, we
are going to use the notion of stable quiver representations. 
In the space of $\mathbb{Z}$-linear functions
$\mathrm{Hom}_{\mathbb{Z}}(\mathbb{Z}Q_0,\mathbb{Z})$ we consider
the basis given by the elements $q^{\ast}$ for $q\in Q_0$, i.e.
$q^{\ast}(q')=\delta_{q,q'}$ for $q\in Q_0$. Define $\dim:=\sum_{q\in
Q_0}q^{\ast}.$ After choosing $\Theta\in
\mathrm{Hom}_{\mathbb{Z}}(\mathbb{Z}Q_0,\mathbb{Z})$, we define the
slope function $\mu:\mathbb{N}Q_0\backslash\{0\}\rightarrow\mathbb{Q}$ via
\[\mu(\alpha)=\frac{\Theta{(\alpha)}}{\dim(\alpha)}.\]
The slope $\mu(\underline{\dim}X)$ of a representation $X$ of $Q$ is
abbreviated to $\mu(X)$.
\begin{defn}
A representation $X$ of $(Q,I)$ is semistable (resp. stable) if for
all proper subrepresentations  $0\neq U\subsetneq X$ the following holds:
\[\mu(U)\leq\mu(X)\text{ (resp. } \mu(U)<\mu(X)).\]
\end{defn}
Denote the set of semistable (resp. stable) points by $R^{ss}_\alpha(Q,I)$ (resp. $R^s_\alpha(Q,I)$).

It is well-known that the definition of $\mu$-stability is equivalent to that of A. King in \cite{kin}. Let $\tilde{\Theta}$ be another linear form. A representation $X$ such that $\tilde{\Theta}(\underline{\dim}X)=0$ is semistable (resp. stable) in the sense of King if and only if
\[\tilde{\Theta}(\underline{\dim} U)\geq 0\,\,(\text{resp.}\,\, \tilde{\Theta}(\underline{\dim} U)>0)\]
for all subrepresentations $U\subset X$ (resp. all proper
subrepresentations $0\neq U\subsetneq X$).\\

In this situation we have the following theorem summarising several main results of \cite{kin}:
\begin{thm}\label{kingsatz}
\begin{enumerate}
\item The set of stable points $R^s_\alpha(Q,I)$ is an open subset of the set of semistable points $R^{ss}_\alpha(Q,I)$, which is an open subset of $R_\alpha(Q,I)$.
\item There exists a categorical quotient $M^{ss}_\alpha(Q,I):=R^{ss}_\alpha(Q,I)//G_{\alpha}$. Moreover, $M^{ss}_\alpha(Q,I)$ is a projective variety.
\item There exists a geometric quotient $M^{s}_\alpha(Q,I):=R_\alpha^{s}(Q,I)/G_\alpha$, which is a smooth subvariety of $M^{ss}_\alpha(Q,I)$.
\end{enumerate}
\end{thm}
\begin{rem}\label{bem1}
\end{rem}
\begin{itemize}
\item The moduli space $M_\alpha^{ss}(Q,I)$ does not parametrize the semistable representations, but the polystable ones.
Polystable representations are such representations which can be
decomposed into stable ones of the same slope, see also \cite{kin}.
\item
For a stable representation $X$ we have that its orbit is of maximal
possible dimension, see \cite{kin}. Since the scalar matrices act trivially on
$R_\alpha(Q,I)$, the isotropy group is one-dimensional. Thus, if the
moduli space is not empty, for the dimension of the moduli space we
have the lower bound
\begin{eqnarray*}\dim M^{s}_\alpha(Q,I)&=&\dim R_\alpha(Q,I)-(\dim G_\alpha-1)\\&\geq& 1-\sum_{q\in Q}\alpha_q^2+\sum_{\rho:q\rightarrow q'\in Q_1}\alpha_q\alpha_{q'}-\sum_{(q,q')\in Q_0\times Q_0}r(q,q',I)\alpha_q\alpha_{q'}.
\end{eqnarray*}
Moreover, if $I=0$ and the moduli space is not empty, we have
\[\dim M^{s}_\alpha(Q)=1-\langle \alpha, \alpha\rangle.\]
\end{itemize}
Finally, we point out some properties of (semi-)stable
representations. These properties will be very useful at different
points of this paper, for proofs see for instance \cite[Section 4]{rei3}.

\begin{lem}\label{zusa} For a bound quiver $(Q,I)$ let $0\rightarrow
Y\rightarrow X\rightarrow Z\rightarrow 0$ be a short exact sequence
of representations.\begin{enumerate}
\item The following are equivalent:
\begin{enumerate}
\item $\mu(Y)\leq\mu(X)$
\item $\mu(X)\leq\mu(Z)$
\item $\mu(Y)\leq\mu(Z)$

\end{enumerate}The same holds when replacing $\leq$ by $<$.
\item The following holds: $\min(\mu(Y),\mu(Z))\leq\mu(X)\leq \max(\mu(Y),\mu(Z)).$
\item If $\mu(Y)=\mu(X)=\mu(Z)$, then $X$ is semistable if and only if $Y$ and $Z$ are semistable.
\item Every stable representation is Schurian.
\end{enumerate}
\end{lem}
If some property is independent of the point chosen in some non-empty open subset $\mathcal O$ of $R_\alpha(Q)$, following \cite{sch}, we say that this property is true for a general representation with dimension vector $\alpha \in\Nn Q_0$.\\ Denote by $\beta\hookrightarrow \alpha$, if a general representation of dimension $\alpha$ has a subrepresentation of dimension $\beta$. A root $\alpha$ is called Schur root if there exists a representation $X$ with $\underline{\dim} X=\alpha$ such that $\End(X)=k$. By \cite[Section 1]{sch} it follows that in this case there already exists an open subset of Schurian representations. From \cite[Theorem 6.1]{sch} we get the following theorem:
\begin{thm}\label{schofield}
Let $\alpha$ be a dimension vector of the quiver $Q$. Then $\alpha$
is a Schur root if and only if for all $\beta\hookrightarrow\alpha$
we have $\Sc{\beta}{\alpha}-\Sc{\alpha}{\beta}>0$.
\end{thm}
Thus, if we define $\Theta_{\alpha}:=\Sc{\rule{0.2cm}{0.4pt}}{\alpha}-\Sc{\alpha}{\rule{0.2cm}{0.4pt}}$, a general representation of dimension $\alpha$ is $\Theta_{\alpha}$-stable in the sense of King if and only if $\alpha$ is a Schur root.\\
Consider the $n$-subspace quiver $S(n)$, i.e.
$S(n)_0=\{q_0,q_1,\ldots,q_n\}$ and $S(n)_1=\{\rho_i:q_i\rightarrow
q_0\mid i\in\{1,\ldots,n\}\}$. Define the slope $\mu$ by choosing
$\Theta=(-1,0,\ldots,0)$. Then we have the following lemma:
\begin{lem}
A representation $X$ of $S(n)$ with dimension vector $\alpha$ is $\mu$-stable
if and only if $X$ is $\Theta_{\alpha}$-stable.
\end{lem}
\begin{proof}
Let $U$ be a subrepresentation of dimension $\beta$. It is easy to
check that we have
\[\frac{-\alpha_{q_0}}{\sum_{i=0}^n\alpha_{q_i}}>\frac{-\beta_{q_0}}{\sum_{i=0}^n\beta_{q_i}}\]
if and only if
\[\Sc{\beta}{\alpha}-\Sc{\alpha}{\beta}=\sum_{i=1}^n\alpha_{q_i}\beta_{q_0}-\sum_{i=1}^n\beta_{q_i}\alpha_{q_0}>0.\]
\end{proof}
\begin{rem}
\end{rem}
\begin{itemize}
\item Note that in general, it is not possible to choose a slope function $\mu$ once and for all such that general representations of all Schur roots are $\mu$-stable. Actually, this is only the case if the rank of the anti-symmetrized adjacency matrix of the quiver
has rank equal to two, see \cite{sti}.
\end{itemize}

\medskip

\noindent We will need the following lemma:
\begin{lem}\label{End}
Let $Y$ and $Z$ be two representations of a bound quiver $(Q,I)$
such that $\Hom(Y,Z)=\Hom(Z,Y)=0$ and $\End(Z)=k$. Let
$\dim\Ext(Z,Y)=d_0>0$. Let $e_1,\ldots,e_d\in\Ext(Z,Y)$ with $1\leq
d\leq d_0$ be linear independent. Consider the exact sequence
\[e:\ses{Y}{X}{Z^d}\] induced by $e_1,\ldots,e_d$. Then we have
$\End(X)\subseteq\End(Y)$.
\end{lem}
\begin{proof}Consider the following long exact sequence
\[
\begin{xy}\xymatrix@R15pt@C20pt{0\ar[r]&\Hom(Z,Y)=0\ar[r]&\Hom(Z,X)\ar[r]&\Hom(Z,Z^d)\ar[r]^{\phi}&\Ext(Z,Y)}\end{xy}
\]
induced by $e$, see \cite[Section A.4]{ass} for more details. By construction $\phi$ is injective and, therefore,
$\Hom(Z,X)=0$. Now consider the following commutative diagram
induced by $e$:
\[
\begin{xy}\xymatrix@R15pt@C20pt{&0\ar[d]&0\ar[d]&0\ar[d]\\0\ar[r]&\Hom(Z^d,Y)=0\ar[d]\ar[r]&\Hom(Z^d,X)\ar[r]\ar[d]&\Hom(Z^d,Z^d)\ar[d]\\
0\ar[r]&\Hom(X,Y)\ar[r]\ar[d]&\Hom(X,X)\ar[r]\ar[d]^{\phi_1}&\Hom(X,Z^d)\ar[d]\\0\ar[r]&\Hom(Y,Y)\ar[r]^{\phi_2}&\Hom(Y,X)\ar[r]&\Hom(Y,Z^d)=0}\end{xy}
\]
Now we also have $\Hom(Z^d,X)=0$. Thus, $\phi_1$ is also injective
and since $\phi_2$ is an isomorphism, the claim follows.
\end{proof}
\noindent Note that the dual lemma dealing with sequences of the
form
\[\ses{Z^d}{X}{Y}\] also holds and can be proven analogously.

\section{Unitarization}
\subsection{Criteria for being unitarizable}
Except for Section \ref{boundquivers}, in the following, we fix the
base field $\mathbb C$. Recall that we may understand a strict representation $X$ of a (bound) quiver $Q(\mathcal{N})$ associated to a poset $\mathcal N$
as a system of vector subspaces $(V_0;(V_q)_{q\in\mathcal{N}})$ and vice versa. We
will use the following criteria for $\chi$-unitarization of $X$
(which was basically obtained by the different authors A.
King \cite{kin}, B. Totaro \cite{tota}, A. Klyachko \cite{kly}, Y. Hu \cite{hu} and others in different formulations).
\begin{thm}\label{huTheorem} Let $(V_0;(V_q)_{q\in\mathcal{N}})$ be an indecomposable strict representation. Then $(V_0;(V_q)_{q\in\mathcal{N}})$ is unitarizable with the weight
$\chi=(\chi_0;(\chi_q)_{q\in\mathcal{N}})\in\Rn_+^{|\mathcal{N}|+1}$
if and only if for every proper subspace $0\neq U\subsetneq V_0$ the
following holds
\begin{gather*}
    \chi_0=\frac{1}{\dim V_0}\sum_{q\in \mathcal N} \chi_q \dim V_q, \\
    \frac{1}{\dim U}\sum_{q\in \mathcal N} \chi_q \dim (V_q\cap
U)<\frac{1}{\dim V_0}\sum_{q\in \mathcal N} \chi_q \dim V_q.
\end{gather*}
\end{thm}

\begin{rem}\label{connection}
\end{rem}
\begin{itemize}
\item If an indecomposable representation $V$ of a poset $\mathcal{N}$ can be unitarized with the weight $\chi \in \mathbb N^{|\mathcal N|+1}$, the corresponding quiver representation $X=F(V)$ is obviously $\tilde{\Theta}$-stable in the sense of King with
\[\tilde{\Theta}=(\chi_0,(-\chi_q)_{q\in \mathcal N})\]
and vice versa. Moreover, choosing the linear form
$\Theta=\mu(X)\dim-\tilde{\Theta}$, where $\mu(X)\in\mathbb{Z}$ can
be chosen arbitrarily,  this representation is $\mu$-stable.
Moreover, we have
\[\chi_0=\mu(X)-\Theta_0, \qquad \chi_q=\Theta_q-\mu(X), \quad q \in \mathcal N.\]
\item It is easy to check that we can modify the linear form $\Theta$ which defines the slope $\mu$ without changing the set of stable points in the following two ways: first we can multiply it by a positive integer; second, we can add an integer multiple of the linear form $\dim$ to $\Theta$. In particular, if we change the linear form appropriately, the weight, which it defines, can be assumed to be positive.
\end{itemize}
We will also use the following lemma:
\begin{lem} \label{extUnit}
    Let $V=(V_0;V_1,\ldots,V_n)$ be an indecomposable $\chi$-unitarizable representation.
    Then for an arbitrary set of subspaces $V_{n+j} \subset V$, $j=1,\ldots,m$, the
    representation $\tilde V=(V_0;V_1,\ldots,$ $V_n,V_{n+1},\ldots,V_{n+m})$ is also
    indecomposable and unitarizable with some weight.
\end{lem}

\begin{proof}
    We prove that $(V_0;V_1,\ldots,V_n,V_{n+1})$ is
    unitarizable with some weight (the remaining part follows by
    induction). Let $U\subset V_0$ be some subspace of $V_0$ such that
    $$
    R=\frac{1}{\dim V_0}\sum_{i=1}^n \chi_i \dim V_i-\frac{1}{\dim U}\sum_{i=1}^n \chi_i \dim (V_i\cap U)
    $$
    is minimal. Note that, it is clear that such a subspace exists because the right hand side only takes finitely many values. Since $V$ is unitarizable and indecomposable, we have $R>0$ and there exists an $\varepsilon>0$ such that
    $R-\varepsilon>0$. Define $\tilde \chi$ in the following way
    $$
        \tilde \chi_i=\chi_i, \quad i=1,\dots,n, \quad
        \tilde \chi_{n+1}=R-\varepsilon.
    $$
   Our claim is that $\tilde V$ is $\tilde\chi$-unitarizable. Indeed,
  let $M\subset V_0$ be some proper subspace of $V_0$ then we have
   \begin{equation*}
    \begin{split}
       \frac{1}{\dim M}\sum_{i=1}^{n+1}\tilde \chi_i \dim
        (V_i\cap M)&=\frac{1}{\dim M}\sum_{i=1}^{n}\chi_i \dim
        (V_i\cap M)+\frac{\tilde\chi_{n+1} \dim(V_{n+1}\cap M)}{\dim M}\\
        &\leq \frac{1}{\dim
        V_0}\sum_{i=1}^{n}\chi_i\dim V_i-R+ \frac{ (R-\varepsilon) \dim(V_{n+1}\cap M)}{\dim
        M}\\
        & < \frac{1}{\dim
        V_0}\sum_{i=1}^{n}\chi_i\dim V_i < \frac{1}{\dim
        V_0}\sum_{i=1}^{n+1}\tilde \chi_i\dim V_i.
    \end{split}
    \end{equation*}
    Hence $(V_0;V_1,\ldots,V_n,V_{n+1})$ is $\tilde\chi$-unitarizable.
\end{proof}

\subsection{Unitarization of general representations of unbound quivers}\label{unitgen}
In this subsection we restrict to posets $\mathcal{N}$ such that the induced quiver $Q(\mathcal N)$ is
unbound, i.e. for the ideal of relations $I$ we have $I=0$. In particular, $Q(\mathcal{N})$ has no oriented and unoriented cycles.
Let $\alpha\in\Nn Q(\mathcal{N})_0$ be a dimension vector. We define
$\varphi_q:N_q\rightarrow \{\pm 1\}$ by
\[\varphi_q(q')=\left\{\begin{array}{l}-1\text{ if } \rho :q'\rightarrow q,\\
1\text{ if }\rho :q\rightarrow q'.\end{array}\right.\] and the
weight $\chi(\alpha)$ by
$$(\chi(\alpha))_q=\left\{\begin{array}{c}
                 \sum_{q'\in N_q}\varphi_q(q')\alpha_{q'}, \quad q\neq q_0 \\
                 -\sum_{q'\in N_q}\varphi_q(q')\alpha_{q'} \quad q=q_0
               \end{array}\right..
$$
This defines a weight function $\chi:\Nn Q(\mathcal{N})_0\rightarrow \Zn Q(\mathcal{N})_0$.
\begin{defn}
Let $\chi:\Nn Q(\mathcal{N})_0\rightarrow \Zn Q(\mathcal{N})_0$ be a weight function and $\alpha\in\Nn Q(\mathcal{N})_0$. If we have $(\chi(\alpha))_q\geq 0$ for every $q\in Q(\mathcal{N})_0$, the dimension vector $\alpha$ is called
$\chi$-positive.
\end{defn}

Note that, if the poset $\mathcal N$ is primitive, then each strict
dimension vector is $\chi$-positive.

\begin{thm}\label{thm1}
\begin{enumerate}\item Let $\alpha$ be a $\chi$-positive Schur root of the unbound quiver
$Q(\mathcal{N})$ induced by a poset $\mathcal{N}$. Then a general representation of $Q(\mathcal{N})$
with dimension vector $\alpha$ can be unitarized with the weight
$\chi(\alpha)$.
\item Let $\alpha$ be a Schur root of the unbound quiver
$Q(\mathcal{N})$ induced by a poset $\mathcal{N}$. Then a general
representation of $Q(\mathcal{N})$ with dimension vector $\alpha$
can be unitarized with a weight $\chi'$ which is obtained by modifying $\chi(\alpha)$ as in Remark \ref{connection}.
\end{enumerate}
\end{thm}
\begin{proof} By $q_0$ we denote the
unique sink. Let $\alpha$ be a Schur root. Let $X$ be a general representation with dimension vector $\alpha$ and let $\beta$ be the dimension vector of a subrepresentation of $X$.
By Theorem \ref{schofield} we have
\[\Theta_{\alpha}(\beta)=\langle\beta,\alpha\rangle-\langle\alpha,\beta\rangle>0.\]
 Then it is easy to check that we have
\begin{eqnarray*}\langle\beta,\alpha\rangle-\langle\alpha,\beta\rangle&=&-\sum_{q\in Q(\mathcal N)_0}
\beta_q\sum_{\substack{q'\in N_q\\\varphi_q(q')=1}}\alpha_q'+\sum_{q\in Q(\mathcal N)_0}\beta_q\sum_{\substack{q'\in N_q\\\varphi_q(q')=-1}}\alpha_q'\\
&=&-\sum_{q\in Q(\mathcal N)_0}\beta_q\sum_{q'\in
N_q}\varphi_q(q')\alpha_{q'}.
\end{eqnarray*}
Recall that $\chi_q(\alpha)=\sum_{q'\in N_q}\varphi_q(q')\alpha_{q'}$. \\By
Theorem \ref{huTheorem}, a representation can be unitarized with the
weight $\chi(\alpha)$ if and only if
\begin{eqnarray*}\frac{1}{\beta_{q_0}}\sum_{q\in Q(\mathcal N)_0 \backslash\{q_0\}}\beta_q\sum_{q'\in N_q}
\varphi_q(q')\alpha_{q'}&<&\frac{1}{\alpha_{q_0}}\sum_{q\in Q(\mathcal N)_0\backslash\{q_0\}}\alpha_q\sum_{q'\in N_q}\varphi_q(q')\alpha_{q'}\\
&=&\frac{1}{\alpha_{q_0}}\sum_{q\in
N_{q_0}}\varphi_{q}(q_0)\alpha_{q_0}\alpha_q=-\sum_{q\in
N_{q_0}}\varphi_{q_0}(q)\alpha_q
\end{eqnarray*}
for all subrepresentations $U$ of dimension vector $\beta$. But
this is obviously the same.\\\\
Taking into account the second part of Remark \ref{connection}, the second part of the
Theorem follows when changing the linear form
$\Theta_{\alpha}$ appropriately.
\end{proof}

\begin{cor} \label{cor_uniq}
Let the $Q({\mathcal{N}})$ induced by the poset $\mathcal{N}$ be
unbound. Then the unique indecomposable representation of a real
root $\alpha$ can be unitarized if and only if $\alpha$ is a real
Schur root.
\end{cor}
\begin{proof}
If $\alpha$ is not a Schur root, we have $\dim\mathrm{End} X_{\alpha}>
1$ for the unique indecomposable representation with dimension
vector $\alpha$. In particular, $X_{\alpha}$ cannot be stable, and
thus cannot be unitarized.\\If $\alpha$ is a real Schur root,
the orbit of $X_{\alpha}$ is dense in the affine variety
$R_{\alpha}(Q)$. Indeed, as already mentioned in Section \ref{rep}, in this case a general representation has trivial endomorphism ring and is, therefore, isomorphic to $X_{\alpha}$. Thus we can apply the preceding theorem.
\end{proof}

\subsection{Unitarization of general representations of bound quivers}\label{boundquivers}
Let $k$ be an algebraically closed field. In this section we state a recipe which can be used to construct stable representations of bound quivers (which are unitarizable for $k=\Cn$). Let $\mathcal{N}$ be a poset and $Q(\mathcal{N})$ be the corresponding (bound) quiver as defined in Section \ref{reprpos}. Note that, in general, we do not have $\Ext_{Q(\mathcal{N})}^i(X,Y)=0$ if $i\geq 2$ for two arbitrary representations $X$ and $Y$ of the quiver $Q(\mathcal{N})$. Thus, in order to obtain some result similar to Theorem \ref{thm1}, the basic idea is the following: we glue polystable representations of an unbound quiver, which is a subquiver of $Q(\mathcal{N})$, with a direct sum of a simple module in order to obtain stable representations of $Q(\mathcal{N})$. Note that the global dimension of the corresponding path algebra of the unbound quiver is one.\\

As already mentioned, we say that a general representation of
dimension $\alpha$ satisfies some property if there exists an open
subset $\mathcal O$ of $R_\alpha(Q)$ such that every representation $X_u\in \mathcal O$, satisfies this property. In abuse of notation, we will
skip the $u$ in what follows. Moreover, if there is more than one
property requested, we always consider elements lying in the
intersection of the corresponding open subsets. In addition, when
considering general representations, we restrict to dimension
vectors whose support can be understood as a quiver without
relations. Recall that otherwise the variety of representations can
be reducible, see \cite{kin}.
Consider 
\begin{align*} 
\nu&:R_{\alpha}(Q)\times R_{\beta}(Q)\rightarrow \mathbb Z,\quad (X,Y)\mapsto\dim\Ext(X,Y).
\end{align*} 
This function is upper semi-continuous, see for instance \cite{sch}. 
By $\mathrm{ext}(\alpha,\beta)$ we denote the minimal value of $\nu$. 

In order to prove the main result of this section, we will frequently make use of the following result \cite[Theorem 3.3]{sch}:   
\begin{thm}\label{sch2}
A general representation of dimension $\alpha$ has a subrepresention of dimension $\beta$ if and only if $\mathrm{ext}(\beta,\alpha-\beta)=0$.
\end{thm}
Thus fixing a dimension vector $\alpha$, we can choose a general representation $X$ in such a way that for every dimension vector $\beta\hookrightarrow\alpha$ there exists a subrepresentation $Y$ of dimension $\beta$ such that $\Ext(Y,X/Y)=0$. Actually, in order to test a representation of dimension $\alpha$ for stability, it is sufficient to consider one subrepresentation for any dimension vector $\beta$ with $\beta\hookrightarrow\alpha$ because the slope only depends on the dimension vector. \\ 

Let $\mathcal{N}$ be a poset corresponding to an unbound quiver and
$\mathcal{M}\subset\mathcal{N}$ be a subset of
elements such that
\[t(\mathcal{M}):=\min\{q\in\mathcal{N}^0\mid
q'\preceq q\,\forall q'\in\mathcal{M}\}\] is unique. If, in addition, $\mathcal M$ is such that for any two elements $q,q'\in \mathcal M$ we have $t(\{q,q'\})=t(\mathcal{M})$ we say that $\mathcal M$ is 
\textit{appropriate}.

\begin{lem}\label{dimform}
Let $\mathcal{M}\subset\mathcal{N}$ be an appropriate subset of $\mathcal N$. Then for a general
representation $X$ of $Q(\mathcal{N})$ we have
\begin{equation*}\dim \bigcap_{q\in\mathcal{M}}X_q=\max\{0,\sum_{q\in \mathcal{M}}\dim X_q-(|\mathcal{M}|-1)\dim X_{t(\mathcal{M})}\}.\end{equation*}
\end{lem}
\begin{proof}
Let $U$ and  $X$ be two $k$-vector spaces such that $U\subsetneq X$ and let $x\in X$. If $(b_1,\ldots,b_{\dim U})$ is a basis of $U$ then $x\in U$ is equivalent to $\mathrm{rank}(b_1,\ldots,b_{\dim U},x)=\dim U$. Thus $x\in U$ is a closed condition because it is equivalent to the vanishing of all $(\dim U+1)$-minors of the defined matrix.

Without loss of generality, we can assume that $\mathcal M\cup t(\mathcal M)$ corresponds to the subspace quiver $S(n)$ for some $n\in\mathbb N$. We proceed by induction on $n$ and on the dimension of $X_{q_n}$ where we use the notation of Section \ref{rep}.
Assume that $0\leq \dim X_{q_n}<\dim X_{q_0}$. If $U:=\bigcap_{i=1}^{n-1}X_{q_i}+X_{q_n}\neq X_{q_0}$, let $x\in X_{q_0}$ such that $x\notin U$. Let $(b_1,\ldots,b_{\dim X_{q_n}})$ be a basis of $X_{q_n}$ and define $\tilde{X}_{q_n}:=\langle b_1,\ldots,b_{\dim X_n},x\rangle$. 

If $U=X_{q_0}$, let $x\notin X_{q_n}$ and define $\tilde{X}_{q_n}$ as before. In both cases we have
\begin{eqnarray*}\dim\bigcap_{i=1}^{n-1}X_{q_i}\cap\tilde{X}_{q_n}&=&\dim\bigcap_{i=1}^{n-1}X_{q_i}+\dim\tilde{X}_{q_n}-\dim (\tilde{X}_{q_n}+\bigcap_{i=1}^{n-1}X_i)\\&=&\max\{0,\sum_{i=1}^{n-1}\dim X_{q_i}+\dim \tilde{X}_{q_n}-(n-1)\dim X_{q_0}\}.\end{eqnarray*}
Thus the claim follows by the first part of the proof. 
\qed
\end{proof}

Let $\mathcal{N}$ be a poset and let 
\[\mathcal{P}=\{q\in\mathcal{N}\mid\exists\, q_1,q_2\in\mathcal{N},\,q_1,q_2\text{ are incomparable}\,,q\prec q_1,q_2\}.\] 
The poset $\mathcal{N'}=\mathcal{N}\backslash\mathcal{P}$ is associated to an unbound quiver. We call the tuple of posets $(\mathcal{N'},\mathcal{N})$ (resp. the tuple of corresponding quivers) related. For instance,
starting with the non-primitive poset $(N,5)$, we get the related primitive
poset $(2,1,5)$, see Section \ref{Section1} for the notation.

In the following, we assume that $\mathcal{N}'$ and $\mathcal{N}=\mathcal{N'}\cup\{q\}$ are related. Moreover, we assume that $N_q$ is an appropriate subset of $\mathcal N$. The first assumption is no restriction because we will see that the case $\mathcal{N}=\mathcal{N'}\cup\{q_1,\ldots,q_n\}$ can be treated by applying Lemma \ref{extUnit}. 

Using the notation of Section \ref{rep} it is easy to check that we have $r(q,t(N_q),I)=|N_q|-1$ and $r(l,l',I)=0$ otherwise where $I$ is the ideal generated by the commutativity relations as described in Section \ref{reprpos}. Fixing a dimension vector, for a representation of the poset $\mathcal{N}$ satisfying the dimension formula of Lemma $\ref{dimform}$, it is often straightforward to write down a projective and injective resolution of minimal length, see \cite[Chapter I.5]{ass} for more details. Moreover, in these cases the global dimension is at most two because projective resolutions of minimal length of the simple modules $S_q$, $q\in Q(\mathcal{N})_0$, defined by $(S_q)_q=k$ and $(S_q)_{q'}=0$ if $q'\neq q$, have at most length two, see \cite[Theorem A.4.8]{ass}.

Obviously, every representation of $Q(\mathcal{N}')$ can be naturally 
understood as a representation of $Q(\mathcal{N})$. Let $\alpha'$ be
a dimension vector of $Q(\mathcal{N'})$ such that a general
representation is polystable with respect to the linear form $\Theta_{\alpha '}$, i.e. the canonical decomposition of $\alpha'$ only consists of
Schur roots of the same slope, see \cite{sch} for the general theory concerning canonical decomposition and \cite{dw} for a very useful algorithm
determining the canonical decomposition. Note that in this case we have $\Sc{\beta}{\alpha'}-\Sc{\alpha'}{\beta}=0$ for all roots $\beta$ contained in the canonical decomposition.

Clearly, every
representation of $Q(\mathcal{N'})$ satisfies the commutativity
relations of $Q(\mathcal{N})$. In particular, the varieties of
representations corresponding to dimension vectors $\alpha\in\Nn
Q(\mathcal{N})_0$ with $\alpha_q=0$ are irreducible, see \cite{kin}. Let
$X'=\bigoplus_{i=1}^m (X_i')^{t_i}$ with $\dim X'=\alpha'$ and
$X'_i\ncong X'_j$ for $i\neq j$ be a general polystable
representation of $Q(\mathcal{N})$ and $S_q$ be the simple module
corresponding to $q$. Since we have $\dim X'_q=0$, it is
straightforward that we have
\[\dim\Ext_{Q(\mathcal{N})}(S_q,X')=\dim \bigcap_{l\in N_q}X'_l=\max\{0,\sum_{l\in N_q}\dim X'_l-(|N_q|-1)\dim X'_{t(N_q)}\}\]
where $t(N_q)$ is the vertex of the quiver $Q(\mathcal{N})$
where the relations starting at $q$
terminate. Thus, if $\dim \bigcap_{l\in N_q}X'_l\neq\{0\}$, we have
\[-\Sc{S_q}{X'}=\sum_{l\in N_q}\dim
X'_l-(|N_q|-1)\dim X'_{t(N_q)}=\dim\Ext_{Q(\mathcal{N})}(S_q,X'),\]
and, therefore, we generally have
$\Ext^2(S_q,X')=0$. Moreover, for two representations $X'$ and $Y'$ of $Q(\mathcal{N'})$ we obviously have $\Ext_{Q(\mathcal{N})}(X',Y')=\Ext_{Q(\mathcal{N'})}(X',Y')$ and $\Hom_{Q(\mathcal{N})}(X',Y')=\Hom_{Q(\mathcal{N'})}(X',Y')$. In the following, we will skip the index $Q(\mathcal{N})$, and we will only use indices if we consider the quiver $Q(\mathcal{N'})$.

\begin{defn}
We call a dimension vector $\alpha'$ of $Q(\mathcal{N'})$ strongly
strict if for a general representation $X'$ with $\underline{\dim}
X'=\alpha'$, we have $\Ext(S_q,X')\neq 0$.
\end{defn}
For instance, in the case of the poset $(N,5)$ we may consider the
related poset $(2,1,5)$ and the unique imaginary Schur root
$\alpha'=(6;2,4;3;1,2,3,4,5)$. This root is strongly strict and we get a representation of the poset $(N,5)$ with dimension vector $\alpha=(6;2,4;1,3;1,2,3,4,5)$ by an extension with the simple module corresponding to the additional source.

Define $n_i:=\dim\Ext(S_q,X'_i)$. Consider the quiver $\tilde{Q}$
with vertices $\tilde{Q}_0=\{l_0,l_1,\ldots,l_m\}$ and arrows
$\tilde{Q}_1=\{\rho_{i,j}:l_0\rightarrow l_j\mid
i\in\{1,\ldots,m\},\,j\in\{1,\ldots,n_i\}\}$. Then every
representation of this quiver with dimension vector
$t=(t_0,t_1,\ldots,t_m)$ induces an exact sequence
$e\in\Ext(S_q^{t_0},X')$ and vice versa. More detailed, keeping in mind the description of exact sequences given in Section \ref{reprpos},
every exact sequence $e$ is uniquely determined by a linear map 
\[f(e):(S_q^{t_0})_q\rightarrow\bigoplus_{i=1}^{m}(\cap_{q'\in N_q}(X'_i)_{q'})^{t_i},\]
i.e. a linear map $f(e):k^{t_0}\rightarrow\bigoplus_{i=1}^mk^{t_in_i}$. In turn, the components of this map define linear maps
\[X(e)_{i,j}=(f(e)_{i,1,j},\ldots,f(e)_{i,t_i,j}):k^{t_0}\rightarrow k^{t_i}\]
for $i=1,\ldots,m$ and $j=1,\ldots,n_i$, and, therefore, a representation of $\tilde{Q}$ of dimension $(t_0,t_1,\ldots,t_m)$. Reversing this construction, every representation of $\tilde{Q}$ defines an exact sequence $e\in\Ext(S_q^{t_0},X')$.
\begin{lem}
The middle terms of $e$ and $e'$ are isomorphic if and only if $X(e)$ and $X(e')$ are isomorphic.
\end{lem} 
\begin{proof}
Let $g\ast X(e)=X(e')$ with $g=(g_0,g_1\ldots,g_m)\in \prod_{i=0}^m\mathrm{Gl}_{t_i}(k)$. Since $\End(X'_i)=k$, this induces bijective endomorphisms $g_0:(S_q)^{t_0}\rightarrow(S_q)^{t_0}$ and $g_i:(X'_i)^{t_i}\rightarrow (X'_i)^{t_i}$ for $i=1,\ldots,m$.
In particular, we get the following commutative diagram
\[
\begin{xy}
\xymatrix@R40pt@C40pt{
(S_q^{t_0})_q\ar[d]^{(X(e)_{i,j})_{j=1,\ldots,n_i}}\ar[r]^{g_0}&(S_q^{t_0})_q\ar[d]^{(X(e')_{i,j})_{j=1,\ldots,n_i}}\\\
(\cap_{q'\in N_q}(X'_i)_{q'})^{t_i}\ar[r]_{g_i|_{\cap_{q'\in N_q}(X'_i)_{q'}}}&(\cap_{q'\in N_q}(X'_i)_{q'})^{t_i}
}
\end{xy}\]

showing that the middle terms associated with $X(e)$ and $X(e')$ are isomorphic.
The other way around, assume that the middle terms of $e$ and $e'$ are isomorphic. Since $\Hom(S_q,X')=\Hom(X',S_q)=0$ and by the universal property of the kernel and cokernel, we naturally obtain a commutative diagram
\[
\begin{xy}
\xymatrix@R20pt@C20pt{
e:0\ar[r]&X'\ar[d]_{g'}\ar[r]&X\ar^\cong[d]\ar[r]&S_q^{t_0}\ar[d]_{g_0}\ar[r]&0\\
e':0\ar[r]&X'\ar[r]&\tilde{X}\ar[r]&S_q^{t_0}\ar[r]&0
}
\end{xy}\]
where $g'$ and $g_0$ are isomorphisms inducing an isomorphism $(g_0,\ldots,g_m)\in\prod_{i=0}^m\mathrm{Gl}_{t_i}(k)$ between $X(e)$ and $X(e')$.  
\end{proof}
We  have the following result, see \cite[Theorem 1.2]{rin} where it is used that the given Schur roots (including the simple one) are pairwise orthogonal, i.e. there exist no homomorphisms between them: 
\begin{thm}
The category of representations of $\tilde{Q}$ is equivalent to the category of representations $X$ of $Q(\mathcal{N})$ having a filtration
\[\ses{\bigoplus_{i=1}^m(X_i')^{t_i}}{X}{S_q^{t_0}}\]
for some $t\in\mathbb{N}\tilde{Q}_0$. 
\end{thm}
We obtain the following corollary:
\begin{cor}Let $X'$ be a polystable representation of $Q(\mathcal{N'})$. Then there exists a strict indecomposable (resp. Schurian) representation of $Q(\mathcal{N})$ satisfying the commutativity relations if $(t_0,t_1,\ldots,t_n)$ is a root (resp. Schur root) of $\tilde{Q}$.
\end{cor}

We call the dimension vector $t$ stable if it is a Schur root and
polystable if it has the canonical decomposition
$t=\bigoplus_{i=1}^m\alpha_{l_i}^{t_l}$ with $\alpha_{l_i}=l_0+n_i l_i$. Moreover, we call the
extension $e$ stable (resp. polystable) if the corresponding
representation of $\tilde{Q}$ is stable (resp. polystable).

We will need the following lemma:
\begin{lem}\label{polystable}
Let $X'=\bigoplus_{i=1}^m (X_i')^{t_i}$ with $X'_i\ncong X'_j$ for
$i\neq j$ be a polystable representation. If $Y'$ is an
indecomposable subrepresentation of $X'$ such that $\Hom(X',Y')\neq 0$, it
follows that $Y'\cong X'_i$ for some $i\in\{1,\ldots,m\}$.
\end{lem}
\begin{proof}
Let $\Psi:X'\rightarrow Y'$ be non-zero. Since the canonical composition $\tau:X'_j\hookrightarrow X'\xrightarrow{\Psi} Y'$ is not zero for some $j$, this defines a factor representation $\mathrm{Im}(\tau)=U$ of $X'_j$ which is a subrepresentation of $X'$. Thus we have $\mu(X'_j)\leq\mu(U)\leq\mu(X')$ and thus $\mu(X'_j)=\mu(U)=\mu(X')$. It follows that $U\cong X'_j$.\\
Since $Y'$ is a subrepresentation of $X'$, the
canonical composition $\phi:Y'\hookrightarrow X'\cong\bigoplus_{i=1}^m (X_i')^{t_i}\twoheadrightarrow
X'_i$ defines a non-zero homomorphism
$\phi\circ\tau:X'_j\rightarrow X'_i$ for some $i$. It follows that $i=j$. Moreover,
since $\End(X'_i)=k$, we obtain that $\phi\circ\tau$ is forced to be an isomorphism and, therefore,
$X'_j$ is a direct summand of $Y'$. Since $Y'$ is indecomposable, we
have $Y'\cong X'_j$.
\end{proof}
In this setup, let $X$ be a stable extension of some general
representation $X'$ with $\underline{\dim} X'=\alpha'$ where
$\alpha'$ is strongly strict. The remaining part of this section is dedicated to proving that we can construct stable (resp. unitarizable) representations of $Q(\mathcal{N})$ in this way, see Theorem \ref{nonprimitive} for the precise statement. 

Every subrepresentation $Y$ of $X$
induces a subrepresentation $Y'$ of $X'$. In particular, we get a
commutative diagram
\begin{eqnarray}\label{cd}
\begin{xy}
\xymatrix@R20pt@C20pt{
&0&0&0\\0\ar[r]&S_q^{r_0}\ar[u]\ar[r]&S_q^{t_0}\ar[u]\ar[r]&S_q^{t_0-r_0}\ar[u]\ar[r]&0\\
0\ar[r]&Y\ar[r]\ar[u]&X\ar[u]\ar[r]&X/Y\ar[u]\ar[r]&0\\
0\ar[r]&Y'\ar[u]\ar[r]&X'\ar[u]\ar[r]&X'/Y'\ar[u]\ar[r]&0\\
&0\ar[u]&0\ar[u]&0\ar[u]&}
\end{xy}
\end{eqnarray}
Since $\alpha'$ is strongly strict, we have
$\Ext^2(Y,X)\cong\Ext^2(Y',X)=0$. Moreover, since
$\Hom(Y',S_q)=\Ext(Y',S_q)=0$, we get $\Hom(Y',X')\cong\Hom(Y',X)$
and $\Ext(Y',X')\cong\Ext(Y',X)$. We also have $\Hom(S_q,X)=0$.
Since we have $\Ext^2(S_q,X')=\Ext^2(S_q,X)=0$, from the long exact
sequence
\[0\rightarrow\Hom(Y,X)\rightarrow\Hom(Y',X)\rightarrow\Ext(S_q^{r_0},X)\rightarrow\Ext(Y,X)\rightarrow\Ext(Y',X)\rightarrow 0\]
we get
\[\dim\Hom(Y,X)-\dim\Ext(Y,X)=\dim\Hom(Y',X')-\dim\Ext(Y',X')-\dim\Ext(S_q^{r_0},X).\]
 Moreover, we have $\Ext^2(X,Y)\cong\Ext^2(X,Y')$ and we get 
long exact sequences
\[0\rightarrow\Hom(X,Y')\rightarrow\Hom(X,Y)\rightarrow\Hom(X,S_q^{r_0})\rightarrow\Ext(X,Y')\rightarrow\Ext(X,Y)\rightarrow 0\]
and
\begin{eqnarray*}0\rightarrow\Hom(X,Y')\rightarrow\Hom(X',Y')\rightarrow\Ext(S_q^{t_0},Y')\rightarrow\Ext(X,Y')\\\rightarrow\Ext(X',Y')\rightarrow\Ext^2(S_q^{t_0},Y')\rightarrow\Ext^2(X,Y')\rightarrow 0\end{eqnarray*}
where $\Ext^2(X',Y')=0$.
Thus we get
\begin{eqnarray*}
\Sc{X}{Y}&=&\dim\Hom(X',Y')-\dim\Ext(X',Y')+\dim\Hom(X,S_q^{r_0})-\dim\Ext(S_q^{t_0},Y')\\&&+\dim\Ext^2(S_q^{t_0},Y').\end{eqnarray*}
Thus in summary we get
\begin{eqnarray*}
\Sc{Y}{X}-\Sc{X}{Y}&=&\Sc{Y'}{X'}-\Sc{X'}{Y'}-\dim\Ext(S_q^{r_0},X)-\dim\Hom(X,S_q^{r_0})\\&&+\dim\Ext(S_q^{t_0},Y')-\dim\Ext^2(S_q^{t_0},Y').
\end{eqnarray*}
For a strongly strict dimension vector $\alpha$ of $Q(\mathcal{N})$
we fix the linear form $\Theta_{\alpha}:\Zn Q_0\rightarrow \Zn$
given by
\begin{eqnarray*}\Theta_{\alpha}(\beta)&=&\Sc{\beta}{\alpha}-\Sc{\alpha}{\beta}\\
&=&-\sum_{\rho:l\rightarrow l'\in
Q(\mathcal{N})_1}\beta_l\alpha_{l'}+\sum_{(l,l')\in
(Q_0)^2}r(l,l',I)\beta_l\alpha_{l'}+ \sum_{\rho:l\rightarrow l'\in
Q(\mathcal{N})_1}\alpha_l\beta_{l'}-\sum_{(l,l')\in
(Q_0)^2}r(l,l',I)\alpha_l\beta_{l'}\\
&=&-\sum_{\rho:l\rightarrow l'\in
Q(\mathcal{N})_1}\beta_{l}\alpha_{l'}+(|N_q|-1)\beta_q\alpha_{t(N_q)}+\sum_{\rho:l\rightarrow l'\in
Q(\mathcal{N})_1}\alpha_{l}\beta_{l'}-(|N_q|-1)\alpha_q\beta_{t(N_q)}.\end{eqnarray*} 
\begin{rem}\label{weight}
\end{rem}
\begin{itemize}
\item Let $\mathcal{N}=\mathcal{N'}\cup \{q\}$ be a poset related to a poset $\mathcal{N'}$. Recall that $\mathcal N'$ is assciated with an unbound quiver. Let $\alpha$ and $\beta$ be two dimension vectors of $\mathcal{N}$ and let $\alpha'$ and $\beta'$ be the corresponding dimension vectors of $\mathcal{N'}$ such that $\beta'\hookrightarrow\alpha '$. Consider the linear form induced by the considerations from above:
\begin{eqnarray*}
\tilde{\Theta}_{\alpha}(\beta)&=&\Sc{\beta'}{\alpha'}-\Sc{\alpha'}{\beta'}-\beta_q(\sum_{l\in
N_q}\alpha_l-(|N_q|-1)\alpha_{t(N_q)})+\\ &&\alpha_q(\sum_{l\in
N_q}\beta_l-(|N_q|-1)\beta_{t(N_q)}).\end{eqnarray*} It is
straightforward that we have
$\Theta_{\alpha}=\tilde{\Theta}_{\alpha}.$ 

This linear form corresponds to the following weight as we will see in the next theorem: let $\chi'$ be the weight as given in Theorem \ref{thm1}. Let $\chi$ be the weight such
$\chi_q=\sum_{l\in N_q}\alpha_l-(|N_q|-1)\alpha_{t(N_q)}$,
$\chi_l=\chi'_l-\alpha_q$ for all $l\in N_q$,
$\chi_{t(N_q)}=\chi'_{t(N_q)}+\alpha_q(|N_q|-1)$ and
$\chi_l=\chi'_l$ for the remaining vertices. 
\end{itemize}
Using the notation of the preceding remark we get the following statement:
\begin{thm} Let $\mathcal{N}=\mathcal{N'}\cup \{q\}$ be a poset related to a poset $\mathcal{N'}$. A representation $X$ of dimension
$\alpha$ can be unitarized with the weight $\chi$ if and only if we have
$\tilde{\Theta}_{\alpha}(\underline{\dim}Y)>0$ for every
subrepresentation $Y$ of $X$. If there is at least one such
representation, there exists an open, not necessarily dense, subset of
unitarizable representations.
\end{thm}
\begin{proof} The first part follows analogously to the proof of
Theorem \ref{thm1}. The second part follows by \cite[Section 1]{sch}. 
\end{proof}
First we will show that every subrepresentation of $X'$ does not
contradict the stability condition.
\begin{prop}\label{Y'}
Let $\alpha$ be $\chi$-positive where the weight is given as in
Remark \ref{weight}. If $r_0=0$, i.e. $Y= Y'$, we have
\[\Sc{Y'}{X}-\Sc{X}{Y'}>0.\]
\end{prop}
\begin{proof}
It is straightforward to check that $\Sc{X'_i}{X}-\Sc{X}{X_i'}=t_0n_i>0$. Thus assume that $Y'$ has no
direct summand isomorphic to $X'_i$ for every $i=1,\dots m$.
Consider the commutative diagram (\ref{cd}).
Since $\Ext(Y',X/Y')\cong\Ext(Y',X'/Y')$ we have $\Ext(Y',X/Y')=0$
by Theorem \ref{sch2}. Moreover, since $X'$ is polystable, by
Lemma \ref{polystable} we get $\Hom(X/Y',Y')\cong\Hom(X'/Y',Y')=0$.
Therefore, we have
\begin{eqnarray*}
\Sc{Y'}{X}-\Sc{X}{Y'}&=&\Sc{Y'}{X/Y'}-\Sc{X/Y'}{Y'}\\
&=&\dim\Hom(Y',X/Y')+\dim\Ext(X/Y',Y')-\dim\Ext^2(X/Y',Y').
\end{eqnarray*}
Define $X_{\cap}:=\bigcap_{l\in N_q}X_l$. First assume that
$Y'_0\cap X_{\cap}\neq 0$. For $l\in N_q$ let $\tau_l$ be the unique
path from $l$ to $t(N_q)$ and $\tau_{t(N_q)}$ be the unique path
from $t(N_q)$ to the root. Let $T_l$ and $T_{t(N_q)}$ respectively
be the representations such that $(T_l)_{l'}=k$ for all $l'$ such that
$l'$ is a tail of some arrow in $\tau_l$ and $(T_l)_{l'}=0$
otherwise. Moreover, we assume $(T_l)_{\rho}=\mathrm{id}$ where it
makes sense. In the same way, we define $T_{t(N_q)}$. Define
$d_0:=\dim Y'_0\cap X_{\cap}-\dim Y'_{t(N_q)}\cap X_{\cap}$. Then we
obtain an exact sequence
\[\ses{Y'}{Y''}{T_{t(N_q)}^{d_0}}\]
where we just glue $T_{t(N_q)}^{d_0}$ to $Y'_0\cap X_{\cap}\neq 0$.
Moreover, if we define $d_l:=\dim Y''_{0}\cap X_{\cap}-\dim
Y''_{l}\cap X_{\cap}$ for all $l\in N_q$, in the same manner we get
\[\ses{Y''}{\tilde{Y}}{\bigoplus_{l\in N_q} T_l^{d_l}}.\]
Note that we obviously have $\tilde{Y}\subset X$. Now by construction we have $\dim\Ext^2(S_q,\tilde{Y})=0$ and,
therefore, $\dim\Ext^2(X/\tilde{Y},\tilde{Y})=0$. Note that we have $\dim\Ext(S_q,\tilde{Y})=\dim\cap_{l\in N_q}
\tilde{Y}_l\neq 0$. Moreover, since the dimension vector is
$\chi$-positive we have
\[\Sc{T_l}{X}-\Sc{X}{T_l}\leq 0\]
for $l\in N_q\cup t(N_q)$. Thus we obtain
\[0<\Sc{\tilde{Y}}{X}-\Sc{X}{\tilde{Y}}=\Sc{Y'}{X}-\Sc{X}{Y'}+\sum_{l\in N_q\cup t(N_q)}(\Sc{T_l}{X}-\Sc{T_l}{X})\leq \Sc{Y'}{X}-\Sc{X}{Y'}.\]
Now assume $Y'_0\cap X_{\cap}=0$. In particular, $S_q$ is no direct
summand of $X'/Y'$. Let $P(q)$ be the indecomposable projective
module corresponding to the vertex $q$ which is given by the vector
spaces $P(q)_l=k$ for all $l\succeq q$ and $P(q)_0=k$ and the
identity map where it makes sense. Now it is straightforward that
the injective dimension of $P(q)$ is one because the cokernel of
$P(q)\hookrightarrow I(0)$ is also injective, where $I(0)$ denotes the indecomposable injective module corresponding to the root. Since
$\Hom(P(q),Y')=0$, there exists a short exact sequence
\[\ses{P(q)^{\dim (X/Y')_q}}{X/Y'}{\overline{X/Y'}}.\]
Since $P(q)$ has injective dimension one and
$\dim(\overline{X/Y'})_q=0$, we have $\Ext^2(X/Y',Y')=0$. Thus the
claim follows.
\end{proof}
The following lemma is used to prove the next proposition:
\begin{lem}\label{ext}
Let a general representation of dimension vector $\alpha'$ be polystable with canonical decomposition $\alpha'=\bigoplus_{i=1}^m (\alpha'_i)^{t_i}$ and let $\beta'\hookrightarrow\alpha'$ such that a general representation of dimension $\beta'$ has no direct summand of dimension $\alpha'_i$ for all $i=1,\ldots,m$. Then there exist general polystable representations $X'=\bigoplus_{i=1}^m (X'_i)^{t_i}$ of dimension $\alpha'$ with $\underline{\dim}X'_i=\alpha'_i$ of 
$\mathcal{N'}$ such that there exists a subrepresentation $Y'$ of $X'$ of dimension $\beta'$ satisfying
\[\dim\Ext(X',Y')\geq \dim\Ext(Y',X').\]
\end{lem}
\begin{proof}
By Theorem \ref{sch2}, we can assume that there exists a subrepresentation $Y'$ of $X'$ such that $\dim\Ext(Y',X'/Y')=0$. Since $\Hom(X',Y')=0$ and $\Ext^2(X'/Y',Y')=0$, the exact sequence
\[\ses{Y'}{X'}{X'/Y'},\]
induces long exact sequences
\[0\rightarrow\Hom(Y',Y')\rightarrow\Ext(X'/Y',Y')\rightarrow\Ext(X',Y')\rightarrow\Ext(Y',Y')\rightarrow 0\]
and
\begin{eqnarray*}
&0\rightarrow\Hom(Y',Y')\rightarrow\Hom(Y',X')\rightarrow\Hom(Y',X'/Y')\rightarrow\Ext(Y',Y')
\rightarrow\Ext(Y',X')\rightarrow 0.\end{eqnarray*} Since we have
$\dim\Ext(Y',X')\leq\dim\Ext(Y',Y')=-\Sc{Y'}{Y'}+\dim\Hom(Y',Y')$,
we get
\begin{eqnarray*}\dim\Ext(X',Y')&=&\dim\Ext(X'/Y',Y')-\Sc{Y'}{Y'}\\
&\geq&\dim\Ext(X'/Y',Y')+\dim\Ext(Y',X')-\dim\Hom(Y',Y').
\end{eqnarray*}
Since $\dim\Ext(X'/Y',Y')\geq\dim\Hom(Y',Y')$, we obtain
\[\dim\Ext(X',Y')\geq \dim\Ext(Y',X').\]
\end{proof}

Next assume that $r_0>0$ and that no direct summand of $Y'$ is a
direct summand of $Y$. Then we have the following:

\begin{prop}\label{stability}
Let a general representation with dimension vector $\alpha'$ be
polystable and let
\[\ses{X'}{X}{S_q^{t_0}}\]
be some stable extension of some general representation $X'$ with $\underline{\dim} X'=\alpha'$. Moreover, assume
$\beta'\hookrightarrow\alpha'$. Let
\[\ses{Y'}{Y}{S_q^{r_0}}\]
with $1\leq r_0\leq t_0$ and $\underline{\dim} Y'=\beta'$ such that $Y$
is a subrepresentation of $X$. Then we have
\[\Sc{Y}{X}-\Sc{X}{Y}>0\]
and $Y$ has no direct summand isomorphic to $S_q$.
\end{prop}
\begin{proof}
The second statement is obvious since the first sequence is stable.

By the construction of Proposition \ref{Y'} we can assume that we generally have $\Ext^2(S_q,Y')=0$. Indeed, otherwise there exists a dimension vector of greater slope such that we generally have $\Ext(S_q,Y')=0$. It follows that $\Ext^2(S_q,Y)=0$, and since we also have $\Ext(X',Y)=0$, we obtain $\Ext^2(X,Y)=0$.

First assume that $Y'\cong \bigoplus(X'_i)^{r_i}$ with $r_i\leq t_i$. Recall that $\Sc{Y'}{X}-\Sc{X}{Y'}=0$ and consider
\[-\dim\Ext(S_q^{r_0},X)-\dim\Hom(X,S_q^{r_0})+\dim\Ext(S_q^{t_0},Y')=-r_0\sum_{i=1}^mt_in_i+t_0\sum_{i=1}^mn_ir_i.\]
Since the extension is stable, we have that $t:=(t_0,t_1,\ldots,t_m)$ is
a Schur root of the quiver $\tilde{Q}$ considered in this section.
Moreover, we have $r:=(r_0,r_1,\ldots,r_m)\hookrightarrow t$. In
particular, we have $\Sc{r}{t}-\Sc{t}{r}>0$. But, this means
\[-r_0\sum_{i=1}^mt_in_i+t_0\sum_{i=1}^mn_ir_i>0.\]

Next assume that $Y'\subsetneq X'$. In particular, we have $\Ext(X'/Y',Y')\neq 0$. Define $X_{\cap}:=\bigcap_{l\in N_q}X'_l$ and define $Y_{\cap}$ analogously. Let $v:=\sum_{i=1}^mn_it_i$ and $e_1,\ldots,e_{v}$ be a basis of $\Ext(S_q,X')$ and let $e=(e_1,\ldots,e_{v})$.
Then we consider the representation $\tilde{X}$ obtained by
\[e:\ses{X'}{\tilde{X}}{S_q^{v}}.\]
Let $\beta\in\Nn Q_0$ be a dimension vector. Consider the quiver Grassmannian $\mathrm{Gr}_{\beta}(\tilde{X})$ of $\tilde{X}$ with dimension vector $\beta$, i.e.
\[\mathrm{Gr}_{\beta}(\tilde{X})=\{N\in\mathrm{Rep}_{\beta}(Q)\mid N\subset \tilde{X}\}\] 
which is a closed (and hence projective) subvariety of $\prod_{l\in Q_0}\mathrm{Gr}_{\beta_l}(\tilde{X}_l)$, see for instance \cite{cr}.  We should mention that quiver Grassmannians are usually defined for quivers without relations. But, since $\tilde{X}$ satisfies the commutativity relations, it is straightforward to check that every subrepresentation satisfies these relations as well.  This is because every linear map associated with $\tilde{X}$ is injective, all arrows are oriented to the unique root and, moreover, we only have at most one arrow between each two vertices. Thus we actually deal with usual quiver Grassmannians.
Consider the natural map $\Pi_q:\mathrm{Gr}_{\beta}(\tilde{X})\hookrightarrow \prod_{l\in Q_0}\mathrm{Gr}_{\beta_l}(\tilde{X}_l)\rightarrow\mathrm{Gr}_{\beta_q}(\tilde{X}_q)$. In particular, $\Pi_q$ is a projective morphism and, thus, it is proper and it follows that $\Pi_q(\mathrm{Gr}_{\beta}(\tilde{X}))$ is closed in $\mathrm{Gr}_{\beta_q}(\tilde{X}_q)$. 

Every subrepresentation $Y'$ of $X'$ corresponds to a subrepresentation $\tilde{Y}$ of $\tilde{X}$ with $\beta_q:=\dim\tilde{Y}_q=\dim Y_\cap$. Let $T_{\tilde{Y}}(\mathrm{Gr}_{\beta}(\tilde{X}))$ be the tangent space at the point corresponding to $\tilde{Y}$.
First assume that $\Pi_q$ is surjective. Then we have $\dim \mathrm{Gr}_{\beta}(\tilde{X})\geq\dim \mathrm{Gr}_{\beta_q}(\tilde{X}_q)=\beta_q(v-\beta_q)$. In particular, we have
$\dim T_{\tilde{Y}}(\mathrm{Gr}_{\beta}(\tilde{X}))\geq\beta_q(v-\beta_q)$ for all $\tilde{Y}\in\mathrm{Gr}_{\beta}(\tilde{X})$. Moreover, by \cite[Proposition 6]{cr} we have
$ T_{\tilde Y}(\mathrm{Gr}_{\beta}(\tilde{X}))\cong\Hom(\tilde{Y},\tilde{X}/\tilde{Y})$.
Obviously, we have $\dim\Hom(Y',X'/Y')\geq\dim\Hom(\tilde{Y},\tilde{X}/\tilde{Y})$.
Since $\beta'\hookrightarrow\alpha'$, applying Theorem \ref{sch2}, we get $\Sc{Y'}{X'/Y'}=\dim\Hom(Y',X'/Y')\geq \beta_q(v-\beta_q)=\dim\Ext(S_q,Y')(v-\dim\Ext(S_q,Y'))$. Since we have $r_0\leq \min\{t_0,\dim\Ext(S_q,Y')\}$, treating the two cases $t_0\leq\dim\Ext(S_q,Y')$ and $t_0>\dim\Ext(S_q,Y')$ separately, it is straightforward to check that
\[\dim\Hom(Y',X'/Y')\geq r_0v-t_0\dim\Ext(S_q,Y')\]
Now the claim follows because $\Ext(Y',X'/Y')\neq 0$.

Next assume that $\Pi_q$ is not surjective. Since the image of $\Pi_q$ is closed and since we deal with a stable extension, we generally have that $\dim \tilde{X}_q\cap Y_{\cap}=\dim Y_\cap+t_0-v\geq r_0.$ Thus we have $\dim\Ext(S^{t_0}_q,Y')\geq t_0(v-t_0+r_0)\geq r_0v$ because $v\geq t_0$. In summary, applying Lemma \ref{ext}, we get
$\Sc{Y}{X}-\Sc{X}{Y}>0$.

\end{proof}

Combining Propositions \ref{Y'} and \ref{stability}, we obtain the following result:
\begin{thm}\label{nonprimitive}
Let $\mathcal{N}=\mathcal{N'}\cup \{q\}$ such that $\mathcal{N}$ and $\mathcal{N}'$ are related. Let a general representation with dimension vector $\alpha'\in \mathbb{N}Q(\mathcal{N}')_0$ be
polystable with respect to
$\Theta_{\alpha'}=\Sc{\rule{0.2cm}{0.4pt}}{\alpha'}-\Sc{\alpha'}{\rule{0.2cm}{0.4pt}}$
and let
\[\ses{X'}{X}{S_q^{t_0}}\]
be some stable extension of some general representation $X'$ with
$\underline{\dim} X'=\alpha'$ such that $\alpha$ is $\chi$-positive
where $\chi$ is given as in Remark \ref{weight}. Then $X$ is stable
and can be unitarized with the weight $\chi$.
\end{thm}
We have the following Corollary:
\begin{cor}\label{corUnit}
Let $\mathcal{N}$ and $\mathcal{N'}$ be related posets such the elements of $\mathcal{N}\backslash\mathcal N'$ are not comparable.
Let a general representation of $Q(\mathcal{N}')$ with dimension vector $\alpha' \in\mathbb{N}Q(\mathcal{N})_0$ be
polystable with respect to $\Theta_{\alpha'}$ and let
\[\ses{X'}{X}{\bigoplus_{q\in\mathcal{N}\backslash\mathcal{N'}} S_q^{t_q}}\]
be an extension of some general representation $X'$ with
$\underline{\dim} X'=\alpha'$. Moreover, let the induced extensions
$e_{q}\in\Ext(S_q^{t_q},X')$ be polystable such that at least one
extension is stable and such that the dimension vector induced by
the stable extension is $\chi$-positive. Then $X$ is stable and can
be unitarized with some weight $\chi$.

\end{cor}
\begin{proof}
We first consider the stable extension
\[\ses{X'}{X''}{{S_q^{t_q}}}.\]
By Theorem \ref{nonprimitive} we have that $X''$ can be unitarized
with the weight as given in Remark \ref{weight}. Now we can apply
Lemma \ref{extUnit} in order to obtain the result.
\end{proof}
It is clear that we can apply Lemma \ref{extUnit}, after having constructed stable representations using the preceding Corollary, in order to construct stable representations of related posets which do not satisfy the condition of the preceding Corollary.
  
Note that, having constructed a stable representation, by Remark
\ref{bem1}, we know a lower bound of the dimension of the moduli
space of stable points.

\section{Unitarizable and non-unitarizable representations of posets}\label{unitnon}

\subsection{Some examples of unitarizable representations}\label{examples}
Using the algorithm provided in Section \ref{boundquivers}, for
instance in tame cases, fixing a dimension vector one can build
families of unitarizable Schurian representations of non-primitive
posets that depend on several complex parameters. Below we provide a
few examples of such posets and dimension vectors. We
start with polystable representations of primitive posets. Then we
glue some subspaces using Corollary \ref{corUnit} in order to
construct Schurian representations, and afterwards we can glue some
extra subspaces as described in Lemma \ref{extUnit}.
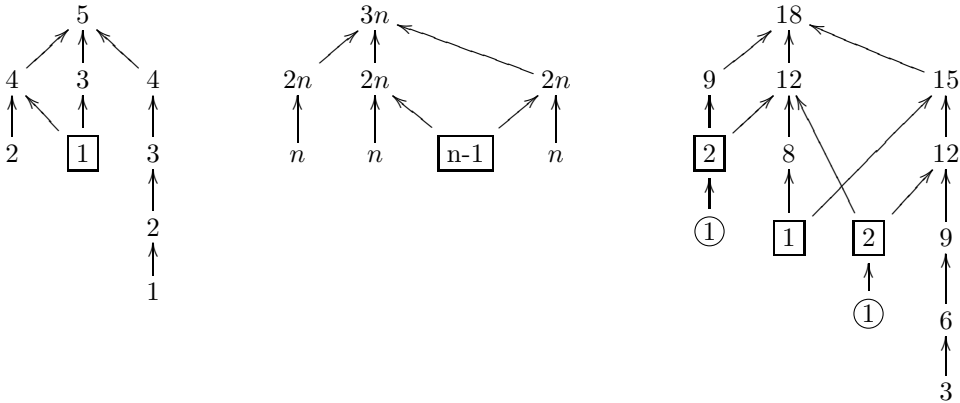
\begin{center}
$
\xymatrix @-1pc {&5 \ar@{<-}[dl] \ar@{<-}[d] \ar@{<-}[dr] & \\
                  4 \ar@{<-}[d] \ar@{<-}[dr] & 3 \ar@{<-}[d]  & 4 \ar@{<-}[d] \\
                  2 & \fbox{1}  & 3 \ar@{<-}[d]  \\
                    &    &  2 \ar@{<-}[d] \\
                    &    &  1 }
                         \qquad \qquad
\xymatrix @-1pc {&3n \ar@{<-}[dl] \ar@{<-}[d] \ar@{<-}[drr] & \\
                  2n \ar@{<-}[d]  & 2n \ar@{<-}[d]\ar@{<-}[dr] &    & 2n \ar@{<-}[d] \ar@{<-}[dl]\\
                  n & n  & \fbox{n-1} & n }
                         \qquad\qquad
\xymatrix @-1pc {&18 \ar@{<-}[dl] \ar@{<-}[d] \ar@{<-}[drr] && \\
                  9  \ar@{<-}[d]  & 12 \ar@{<-}[d] \ar@{<-}[dl] \ar@{<-}[ddr] &    & 15 \ar@{<-}[ddll] \ar@{<-}[d]\\
                  \fbox{2} \ar@{<-}[d] & 8\ar@{<-}[d] &  & 12  \ar@{<-}[d]\ar@{<-}[dl] \\
                   \mycirc{1} & \fbox{1}  & \fbox{2} \ar@{<-}[d] & 9 \ar@{<-}[d] \\
                    & & \mycirc{1}  & 6 \ar@{<-}[d] \\
                    &  &  & 3 }$
\end{center}
Here by $\fbox{i}$ we denote those elements that are glued to the
primitive posets as in Corollary \ref{corUnit}, and by $\mycirc{j}$
we denote elements glued to stable representation using Lemma
\ref{extUnit}.  It is clear that one can produce many of such
examples. Notice that for some posets and their dimension vectors
the provided technique is not applicable, for example in the
following cases
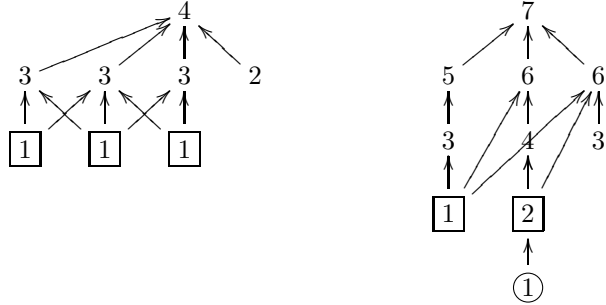
\begin{center}
$
\xymatrix @-1pc {& &  \ar@{<-}[dl]  \ar@{<-}[dll] 4  \ar@{<-}[d]  \ar@{<-}[dr] & & \\
                 3 \ar@{<-}[dr] \ar@{<-}[d] & 3 \ar@{<-}[dl] \ar@{<-}[dr] \ar@{<-}[d] & 3\ar@{<-}[dl] \ar@{<-}[d] & 2\\
               \fbox{1} &  \fbox{1} &   \fbox{1} &} \qquad
               \qquad
\xymatrix @-1pc {&7 \ar@{<-}[dl] \ar@{<-}[d] \ar@{<-}[dr] & \\
                  5 \ar@{<-}[d] & 6 \ar@{<-}[d] \ar@{<-}[ddl] & 6 \ar@{<-}[d] \ar@{<-}[ddll] \\
                  3  \ar@{<-}[d] & 4 \ar@{<-}[d]   \ar@{<-}[d]  & 3   \\
                  \fbox{1} &  \fbox{2}  \ar@{<-}[d] \ar@{->}[uur]  & \\
                   &  \mycirc{1} & \\ } $
\end{center}
In these cases the corresponding representations of the primitive
posets are not polystable because the canonical decompositions of
the dimension vectors are $(4;3;3;3;2)=(3;2;2;2;2)\oplus(1;1;1;1;0)$
and $(7;3,5;4,6;3,6)=(2;1,1;1,2;1,2)\oplus (1;0,1;1,1;0,1)\oplus
(4;2,3;2,3;2,3)$.

\begin{rem}
\end{rem}
\begin{itemize}

\item An interesting question is whether the stability condition
$\Theta_{\alpha}=\Sc{\rule{0.2cm}{0.4pt}}{\alpha}-\Sc{\alpha}{\rule{0.2cm}{0.4pt}}$
determines a dense subset of Schurian representations of the poset
$\mathcal N$ with dimension $\alpha$ (an analogue of Schofield's
Theorem \ref{schofield}). Notice that it is straightforward to check
that if $X$ is an indecomposable quite sincere representation (i.e. $0\neq \dim X_q\neq \dim X_0, \ q \in \mathcal N$, and $\dim X_{q'}<\dim X_q$ if $q'\prec q$) of a poset of
representation finite type then $X$ is stable with the weight
$\Theta_{\underline{\dim} X}$ (see also \cite{GrushevoyYusenko}).

\end{itemize}

\subsection{Unitarization of rigid modules}

The \emph{rigidity index} (see for example \cite{Katz96})
$\mathrm{rig}(A_i)$ of a collection $(A_1,\ldots,A_n)$ of matrices $A_i\in M_m(\mathbb C)$  is defined by
    $$
        \mathrm{rig}(A_1,\ldots,A_n)=m^2(2-n)+\sum\limits_{i=1}^{n}\dim(Z(A_i)),
    $$
    where $Z(X)$ denotes the commutator of the matrix $X$, i.e.
    \[Z(X)=\{A\in M_m(\mathbb C)\ |\ AX=XA\}.\]

N. Katz, see \cite{Katz96}, showed that if $(A_1,\ldots,A_n)$ is an
irreducible system of matrices satisfying $$A_1\cdot\ldots\cdot A_n=I$$ then
$$\mathrm{rig}(A_1,\ldots,A_n)\in\{2j\mid j\in\Zn,\ j\leq 1\}.$$  
Following Katz we say that the set of matrices is
\emph{rigid} if its rigidity index equals 2, otherwise we say that
the set of matrices is \emph{non-rigid}. 

\begin{lem} \label{commLemma}
    Let $A \in M_m(\mathbb C)$ be an arbitrary Hermitian matrix with
    eigenvalues $\{\lambda_i\}_{i=1}^j$. Let the multiplicity of each
    $\lambda_i$ be $d_i$. Then
    $$
        \dim Z(A)=d_1^2+\ldots+d_j^2.
    $$
\end{lem}
\begin{proof}
    It is clear that the dimension of the commutator of the matrix $A$ does not depend on the representative of
    the conjugacy class of $A$. Hence we can assume that
    $$A=\mathrm{diag}\{\lambda_1,\ldots,\lambda_1,\ldots,\lambda_j,\ldots,\lambda_j\}.$$
    Then $Z(A)=M_{d_1}(\mathbb C)\oplus\ldots\oplus
    M_{d_j}(\mathbb
    C)$. Now the statement is obvious.
\end{proof}
  
Recall that a module $X$ is called \emph{rigid} if
$\textrm{Ext}(X,X)=0.$

\begin{thm} \label{rigidprimitive}
    Let $\mathcal N$ be a primitive poset of type $(m_1,\ldots,m_n)$. Assume that $X$ is an indecomposable rigid strict representation of $Q(\mathcal N)$. Then it is unitarizable with some weight $\chi$ and the corresponding representation of the algebra $\mathcal B_{\Gamma(\mathcal N),\omega}$, viewed as collection  of Hermitian matrices $A_1,\ldots,A_n$, is rigid.
\end{thm}

\begin{proof}

By Corollary \ref{cor_uniq} the unique indecomposable representation
$X$ of a real Schur root can be unitarized. This representation is
\emph{rigid} due to $\dim\End(X)-\dim\Ext(X,X)=1$. In this case
Schofield's Theorem \ref{schofield} can be checked easily, see also
\cite[Lemma 5.1]{hp}.
    We take the Euler characteristic for $X$, i.e.
    $$
       \Sc{ X}{ X}=\dim \textrm{End}(X)-\dim \textrm{Ext}(X,X)=\sum_{q\in Q_0} \dim X_q \dim X_q - \sum_{\rho:{q\rightarrow q'}} \dim X_q \dim X_q'=1.
    $$
    Using Lemma \ref{commLemma} we have that
    \begin{equation*}
    \begin{split}
        \mathrm{rig}(A_1,\dots,A_n)&=m^2(2-n)+\sum\limits_{i=1}^{n}\dim(Z(A_i))=m^2(2-n)+\sum\limits_{i=1}^{n}\sum\limits_{j=1}^{m_i+1}{d^{(i)}_j}^2,
    \end{split}
    \end{equation*}
    where $m=\dim X_0$ and $d^{(i)}_j$ is the dimension of the $j$-th eigenspace of the
    corresponding Hermitian matrix $A_i$, which are given by
    \begin{equation*}
        \begin{split}
            d_{1}^{(i)}=\dim X_{1}^{(i)},\quad d_{j}^{(i)}&=\dim X_{j}^{(i)}-\dim X_{j-1}^{(i)}, \quad 2\leq j
            \leq m_i,\\
            d_{m_i+1}^{(i)}&=\dim X_0-\dim X^{(i)}_{m_i}.
        \end{split}
    \end{equation*}
    Then taking $\Sc{X}{X}$ we get
    \begin{equation*}
        \begin{split}
            \Sc{X}{X}&=(\dim X_0)^2+\sum_{i=1}^{n}\sum_{j=1}^{m_i}(\dim X_{j}^{(i)})^2\\
            &-\sum_{i=1}^{n}\sum\limits_{j=1}^{m_i-1}(\dim X_{j}^{(i)})(\dim X_{j+1}^{(i)})-
            \sum_{i=1}^{n} (\dim X_0)(\dim X_{m_i}^{(i)})\\
            &=(\dim {X}_0)^2 + \frac{1}{2}\sum_{i=1}^{n}\left((\dim X_{1}^{(i)})^2+\sum\limits_{j=2}^{m_i}
            ((\dim X_{j}^{(i)})-(\dim X_{j-1}^{(i)}))^2 \right)\\
            &+ \frac{1}{2}\sum_{i=1}^{n}
            ((\dim {X}_{0})-(\dim X_{m_i}^{(i)}))^2-\frac{n}{2}(\dim {X}_0)^2\\
            &=\frac{1}{2}(m^2(2-n)+\sum\limits_{i=1}^{n}\sum\limits_{j=1}^{m_i+1}{d^{(i)}_j}^2)=
            \frac{1}{2}\mathrm{rig}(A_1,\ldots,A_n).
        \end{split}
    \end{equation*}
    Since $\Sc{X}{X}=1$ because $X$ is rigid, the corresponding set of the matrices is also rigid.
\end{proof}

Let $\mathcal N$ be a non-primitive poset, and let $\mathcal N'$ be a
related primitive poset, i.e. $\mathcal N=\mathcal{N'}\cup\{q_1,\ldots,q_n\}$.
The following Corollary is straightforward:

\begin{cor} \label{rigidNonprimitive}
    Assume that $X$ is a rigid Schurian representation of $\mathcal N$
    such that the following condition holds
    $$
        \dim X_{q_i} \leq \sum_{l \in N_{q_i}}
        \dim X_l-(|N_{q_i}|-1)\dim X_{t(N_{q_i})}
    $$
    for all $q_i$ and that the corresponding representation of the related primitive poset is Schurian. Then $X$ can be unitarized with some  weight.
\end{cor}

\begin{proof}
    It is straightforward to see that the corresponding
    representation $X'$ of $\mathcal N'$ is rigid.
    Then we can
    apply Proposition \ref{rigidprimitive} and Corollary \ref{corUnit}
    to obtain the statement.
\end{proof}

\subsection{ADE classification of unitarizable representations}

\begin{thm}\label{ade} Let $Q(\mathcal{N})$ be an unbound quiver induced by a  poset $\mathcal{N}$. Then we have:
\begin{enumerate}
\item Every indecomposable strict representation of $Q(\mathcal{N})$ is
unitarizable if and only if $Q(\mathcal{N})$ is a Dynkin quiver.

\item Every Schurian strict representation of $Q(\mathcal{N})$ is
unitarizable if and only if $Q(\mathcal{N})$ is a subquiver of an extended Dynkin
quiver.

\item There exist families of non-isomorphic unitarizable and non-unitarizable Schurian strict representations which depend on
arbitrary many continuous parameters if and only if $Q(\mathcal{N})$ contains an extended Dynkin quiver as a proper subquiver.
\end{enumerate}
\end{thm}

\begin{proof}

\emph{The first part} trivially follows from the previous section
observing that in this case all indecomposable representations are
Schurian and rigid. Moreover, if the underlying quiver is not of
Dynkin type, there always exist non Schurian roots. Indeed, we may
consider an isotropic root $\alpha$, i.e. $\langle
\alpha,\alpha\rangle=0$. Now it is easy to check that $2\alpha$ is
no Schur root, but a root, since the canonical decomposition of
$2\alpha$ is $\alpha\oplus\alpha$, see also \cite{sch}.

\emph{Second part.} The representations that correspond to real
Schur roots are obviously rigid and hence unitarizable. In general, by
\cite[Proposition 5.2]{hp} any Schur representation is stable for
some linear form $\Theta$. Thus following Remark \ref{connection} it
can be unitarized with some weight. Let us notice that this result
(together with the description of possible weights) was
alternatively obtained in the series of D.Yakimenko's papers (see \cite{yak} and references therein).

{\it Third part.}  Let $\alpha$ be an indivisible isotropic Schur
root of an extended Dynkin quiver. Thus a general representation $X$
with dimension vector $\alpha$ is Schurian and can be unitarized by
Theorem \ref{thm1}. By adding an extra vertex with fixed dimension
$d$ to a vertex $q$ with $\dim X_q>d\geq 1$ to the extended Dynkin
quiver we again get a Schurian representation, say with dimension
vector $\tilde{\alpha}$. In particular, $\tilde{\alpha}$ is a Schur
root. Indeed, we may for instance apply Lemma \ref{End} in order to
see that the new representation is a Schurian representation.

It is easy to check that $\langle \tilde{\alpha},
\tilde{\alpha}\rangle=\langle\alpha,\alpha\rangle+d^2-d\alpha_q<0$.
For two general stable representations $X$ and $Y$ of dimension
$\tilde \alpha$ we have $\mathrm{Hom}(X,Y)=\mathrm{Hom}(Y,X)=0$. Let
\[\ses{X}{Z}{Y}\] be a non-splitting exact sequence. Then by Lemma
\ref{End}, we have $\End(Z)\subseteq\End(X)=\Cn$. Thus, $Z$ is a
semistable Schurian representation which is not stable. Now we can
check by a direct calculation that for every weight $\chi$ we have
that $X$ is a subrepresentation which contradicts $\chi$-stability,
see Lemma \ref{zusa}.

If we want to glue a vertex $q$ to some vertex $q'$ of dimension one
we proceed as follows: first we add an extra arrow
$\rho:q'\rightarrow q$ and consider some non-splitting exact
sequence $\ses{S_q}{Z}{X\oplus X'}$ where $S_q$ is the simple module
corresponding to the vertex $q$ and $X$ and $X'$ are non-isomorphic
Schurian of dimension $\alpha$, thus $\Hom(X,X')=0$. Then, applying
Lemma \ref{End} to the induced sequences $\ses{S_q}{Z'}{X}$ and
$\ses{Z'}{Z}{X'}$ we obtain $\End(Z)=\Cn$. It is easy to check that
$\dim Z$ is an imaginary root which is not isotropic. Now by
applying the reflection functor, see \cite{bgp}, corresponding to
the vertex $q$ we again get a Schurian representation $\tilde{Z}$.
But $\tilde{Z}$ corresponds to some filtration and we can proceed as
in the first case.

The existence of a family of unitarizable representations depending on an arbitrary number of parameters follows in the same manner as Theorem \ref{lastthm}.
\end{proof}

\begin{rem} \label{bem}
\end{rem}
\begin{itemize}
\item
    Let us remark that the first and third part of the theorem hold
    for posets in general.

    \noindent If the poset is
    of representation finite type, then each indecomposable representation can
    be unitarized with some weight (see \cite{GrushevoyYusenko} for
    the proof).

    \noindent     If the poset contains a poset of wild type as a subposet, the same argument as for unbound quivers can be applied. Thus, there is a family of non-isomorphic non-unitarizable
    Schurian representations of the poset that depends on arbitrary many continuous
    parameters.

    \noindent But it is an open question whether all Schurian representations of tame posets with unoriented cycles are unitarizable. Like in Section \ref{examples}, in many cases it is possible to construct an open subset of unitarizable representations. But as in the case without cycles the constructed weight does not apply for all Schurian representations.
\end{itemize}

\section{Complexity of the description of $*$-representations of
$\mathcal A_{\mathcal N,\chi}$}\label{kap}

\begin{thm}\label{lastthm}

Let $\mathcal N$ be a poset of representation wild type. Then it is possible to
choose the weight $\chi_{\mathcal N}$ in such a way that for an
arbitrary natural number $n$ there exists a family of non-isomorphic
Schurian representations of $\mathcal N$ depending on at least $n$
complex parameters which can be unitarized with the weight
$\chi_{\mathcal N}$.
\end{thm}

\begin{proof}
Due to Theorem \ref{reptype} we only need to consider critical posets of the 
following types: $(1,1,1,1,1)$, $(1,1,1,2)$, $(2,2,3)$, $(1,3,4)$, $(1,2,6)$ and 
$(N,5)$. Let us consider the
dimension vectors $(2;1;1;1;1;1)$, $(4;2;2;2;1,2)$,
$(6;2,4;2,4;1,2,4)$, $(8;4;2,4,6;1,2,4,6)$ and
$(12;6;4,8;1,2,4,6,8,10)$ respectively of the quivers corresponding
to $(1,1,1,1,1)$, $(1,1,1,2)$, $(2,2,3)$, $(1,3,4)$ and $(1,2,6)$
respectively. In order to see that these are Schur roots, we can
easily construct a Schurian representation of these dimension
vectors. For instance for $(4;2;2;2;1,2)$ we consider a
non-splitting short exact sequence $\ses{X'\oplus X}{Y}{S_5}$ where
$X$ and $X'$ are Schurian representations of dimension vector
$(2;1;1;1;0,1)$ with $\Hom(X,X')=0$. As in the proof of Theorem
\ref{ade} we may apply Lemma $\ref{End}$ to the two induced
sequences. The other cases behave analogously.

For the first five dimension vectors $\alpha_i$ we have that
$\Sc{\alpha_i}{\alpha_i}=-1$. Hence, following Remark \ref{bem1} it
is possible to choose a two parameter family of Schurian
representations. By Theorem \ref{thm1}, a general representation
with this dimension vector can be unitarized with the weights
\begin{align*}
    \chi_{(1,1,1,1,1)}&=(5;2;2;2;2;2),\\
    \chi_{(1,1,1,2)}&=(8;4;4;4;2,3),\\
    \chi_{(2,2,3)}&=(12;4,4;4,4;2,3,4),\\
    \chi_{(1,3,4)}&=(16;8;4,4,4;2,3,4,4),\\
    \chi_{(1,2,6)}&=(24;12;8,8;2,3,4,4,4,4).\\
    \end{align*}
For two general unitarizable representations $X$ and $X'$ of
dimension $\alpha$ we have
$\mathrm{Hom}(X,X')=\mathrm{Hom}(X',X)=0$. The middle term $Z$ of
every non-splitting exact sequence
\[\ses{X}{Z}{X'}\]
has dimension vector $2\alpha$. Moreover, such representations, even
if they are not stable, are Schurian by Lemma \ref{End}. Thus there
exists a non-empty open subset of Schurian representations having
the same dimension vector. Following Theorems \ref{schofield} and
\ref{thm1}, there exists a non-empty open subset of representations
which can be unitarized with the weight $2\chi_{\mathcal N}$ and,
therefore, with the weight $\chi_{\mathcal N}$, too. The dimension
of the corresponding moduli space is $1-\Sc{\dim Z}{\dim
Z}=1-\Sc{2\alpha}{2\alpha}=1-4\Sc{\alpha}{\alpha}=5$, see Remark
\ref{bem1}. Hence there exists a $5$-parameters family of
non-isomorphic representations having dimension vector $2\alpha$
which can be unitarized with the same weight $\chi_{\mathcal N}$.
Then iterating the same procedure for the dimension vector
$2\alpha$, we will obtain the desirable result due to the fact that
$1-\Sc{2^n\alpha}{2^n\alpha}=1+2^{2n}$ growths when iterating.

In the case of the poset $(N,5)$ we proceed as follows. We consider
the related poset $(2,1,5)$ and the isotropic Schur root
$\alpha=(6;2,4;3;1,2,3,4,5)$. This root is strongly strict. Taking a
general polystable representation $X'=\bigoplus_{i=1}^m (X'_i)^{t_i}$
with $\underline{\dim} X'_i=\alpha$ and $X'_i\ncong X'_j$, we can
consider stable extensions
\[\ses{X'}{X}{S^{t_0}_q}.\]
where $(t_0,t_1,\ldots,t_m)$ is a Schur root of the
dual quiver of the $m$-subspace quiver $S(m)$. Note that the intersection of the two
questioned subspaces is of dimension one. By Theorem
\ref{nonprimitive} any such representation $X$ is stable, i.e. in
particular Schurian, and can be unitarized (for instance with the weight
$\chi_{(N,5)}=(11;4,3;1,5;2,2,2,2,2)$ for $(t_0,t_1)=(1,1)$). Since $\alpha$ is an
isotropic root we have a one-parameter family of stable
representations of dimension $\alpha$. In particular, we have a
$d$-parameter family of polystable representations for every tuple
$(t_1,\ldots,t_m)$ with $\sum_{i=1}^mt_i=d$.
\end{proof}

\begin{cor}

Let $\Gamma$ be a star-shaped graph that contains an extended Dynkin
graph as a proper subgraph. Then there exists a character
$\omega_\Gamma$ such that the algebra $\mathcal
B_{\Gamma,\omega_\Gamma}$ has a family of unitary-nonequivalent
irreducible $*$-representations which depends on an arbitrary number
of continuous parameters.
\end{cor}

\begin{proof}
Using the previous theorem and the relations between unitarizable
systems of subspaces and $*$-representations of $\mathcal
B_{\Gamma,\omega_\Gamma}$ it is easy to check that letting
\begin{align*}
    \omega_{(1,1,1,1,1)}&=(5;2;2;2;2;2),\\
    \omega_{(1,1,1,2)}&=(8;4;4;4;2,5),\\
    \omega_{(2,2,3)}&=(12;4,8;4,8;2,5,9),\\
    \omega_{(1,3,4)}&=(16;8;4,8,12;2,5,9,13),\\
    \omega_{(1,2,6)}&=(24;12;8,16;2,5,9,13,17);
\end{align*}
we obtain the desirable statement.
\end{proof}

\medskip

\begin{rem}\label{bem2}
\end{rem}
\begin{itemize}
\item
In \cite{KrugNazRoi2} it was conjectured that if the graph $\Gamma$ contains an extended Dynkin graph as a proper subgraph, then there exists a characters $\omega_\Gamma$ such that the algebra $\mathcal B_{\Gamma,\omega_\Gamma}$ is $*$-wild. The previous Corollary gives possible candidates for $\omega_\Gamma$ among all possible characters, since it is obvious that for such characters the classification task is an extremely difficult problem.  In the case when  $\Gamma=(1,1,1,1,1)$ or $\Gamma=(1,1,1,2)$ it is known that the algebras 
$\mathcal B_{\Gamma,\omega_\Gamma}$ are indeed $*$-wild with the characters $\omega_\Gamma$ given as in the previous Corollary. This is due to \cite[Section 3.1.3]{OstrovskyiSamoilenko} and private communication with S. Rabanovich. The other cases are unknown by now.  
\end{itemize}



\end{document}